\documentclass{article}
\usepackage{graphicx}
\usepackage{authblk}
\usepackage{amsmath, amsthm}

\usepackage{eulerpx}

\usepackage{bm}
\newcommand{\mathbbm}[1]{\text{\usefont{U}{bbm}{m}{n}#1}} 

\newtheorem{theorem}{Theorem}
\newtheorem{proposition}{Proposition}

\usepackage[overload]{empheq}

\usepackage{graphicx, grffile}
\graphicspath{{images/}}
\usepackage[a4paper,left=3cm,right=3cm]{geometry}
\usepackage{subfig}

\usepackage{tikz}
\usetikzlibrary{}

\usepackage{tabularx}

\usepackage{algorithm2e}

\providecommand{\keywords}[1]
{
  \small	
  \textbf{Keywords:} #1
}

\usepackage[colorlinks, citecolor=black, linkcolor=black, urlcolor=black, linktoc=all]{hyperref}
\usepackage{cleveref}
\usepackage[numbered]{bookmark} 

\newcommand{\abs}[1]{\left|#1\right|}

\newtheorem{hypothesis}{Hypothesis}
\newtheorem{lemma}[subsubsection]{Lemma}

\newtheorem{definition}[subsubsection]{Definition}
\newtheorem{remark}[subsubsection]{Remark}

\begin{document}

\title{The Averaging Principle for Non-autonomous Slow-fast Stochastic Differential Equations and an Application to a Local Stochastic Volatility Model}
\author{Filippo de Feo \\ filippo.defeo@polimi.it}

\affil{Department of Mathematics, Politecnico di Milano, \\ Piazza Leonardo da Vinci 32, 20133 Milano,
Italy }

\maketitle
\begin{abstract}
    In this work we study the averaging principle for non-autonomous slow-fast systems of stochastic differential equations. In particular in the first part we prove the averaging principle assuming the sublinearity, the Lipschitzianity and the Holder's continuity in time of the coefficients, an ergodic hypothesis and an $\mathcal{L}^2$-bound of the fast component. In this setting we prove the weak convergence of the slow component to the solution of the averaged equation. Moreover we provide a suitable dissipativity condition under which the ergodic hypothesis and the $\mathcal{L}^2$-bound of the fast component, which are implicit conditions, are satisfied.
    
 In the second part we propose a financial application of this result: we apply the theory developed to a slow-fast local stochastic volatility model. First we prove the weak convergence of the model to a local volatility one. Then under a risk neutral measure we show that the prices of the derivatives, possibly path-dependent, converge to the ones calculated using the limit model.
\end{abstract}


\section{Introduction}
In this work we are concerned with the study of the averaging principle for slow-fast systems of stochastic differential equations of the form
\begin{equation}\label{eq:synopsis}
\begin{cases}
dX_t^{\epsilon}=b(t,X_t^{\epsilon},Y_t^{\epsilon})dt+\sigma(t,X_t^{\epsilon},Y_t^{\epsilon})dW_t \\
X_0^{\epsilon}=x_0  \in \mathbb{R}^d \\
dY_t^{\epsilon}=\epsilon^{-1} B(t,X_t^{\epsilon},Y_t^{\epsilon})dt+\epsilon^{-1/2}C(t,X_t^{\epsilon},Y_t^{\epsilon})d\tilde{W}_t \\
Y_0^{\epsilon}=y_0 \in \mathbb{R}^l
\end{cases}
\end{equation}
where t $\in [0,T]$, W, $\tilde{W}$ are Brownian motions, $b$, $\sigma$, $B$, $C$ are some regular functions and $\epsilon \in (0,1]$ is a small parameter representing the ratio of time-scales between the slow component $X^{\epsilon}$ and the fast component $Y^{\epsilon}$.\\
The idea of the averaging principle is to show the convergence of $X^{\epsilon}$ as $\epsilon \to 0$ to the solution $\overline{X}$ of the so called averaged equation
 \begin{align*} 
d\overline{X} _t = \overline{b}(t,\overline{X}_t)dt+\overline{\sigma}(t,\overline{X}_t)dW_t && \overline{X}_0=x_0, \forall t \in[0,T]
\end{align*}
which doesn't depend anymore on the fast component and the coefficients $\overline{b}$, $\overline{\sigma}$ are opportune averages of the original ones b, $\sigma$.\\
In this sense from a mathematical point of view it is possible to model systems with different time-scales and then operate a rigorous dimensionality reduction, approximating the behavior of the slow component $X^{\epsilon}$ with $\overline{X}$ and controlling the error of such approximation.
For this reason multi-scale stochastic systems are widely used in many areas of physics, chemistry, biology, financial mathematics and many other applications areas (e.g. \cite{papanicolaou2}, \cite{papanicolaou3}, \cite{rawlings}, \cite{WEINAN}).

The averaging principle was initially proposed for the classical equations of celestial mechanics by Clairaut, Laplace and Lagrange in the eighteenth century. Later on in the twentieth century rigorous result were proved by Bogoliubov, Krylov and Mitropolsky. In \cite[Appendices]{Sanders} it can be found an entire section devoted to the history of the averaging principle for differential equations.\\
For what concerns stochastic differential equations the first results on the averaging principle were due to Khasminskii, Vrkoč and Gikhman in the '60s.\\
In \cite{Khasminskii} Khasminskii considered the following autonomous system of stochastic differential equations 
 \begin{equation*}
\begin{cases}
dX_t^{\epsilon}=b(X_t^{\epsilon},Y_t^{\epsilon})dt+\sigma(X_t^{\epsilon},Y_t^{\epsilon})dW_t \\
X_0^{\epsilon}=x_0  \in \mathbb{R}^d \\
dY_t^{\epsilon}=\epsilon^{-1} B(X_t^{\epsilon},Y_t^{\epsilon})dt+\epsilon^{-1/2}C(X_t^{\epsilon},Y_t^{\epsilon})d\tilde{W}_t \\
Y_0^{\epsilon}=y_0 \in \mathbb{R}^l
\end{cases}
\end{equation*}
and proved that the averaging principle holds under Lipschitz coefficients and $b,\sigma$ uniformly bounded in $y$. Moreover  a bound of the fast component and the following ergodic condition are assumed:\\
there exist a d-dimensional vector $\overline{b}(x)$ and a square d-dimensional matrix $\overline{a}(x)$ such that for every $t_0 \leq T$, $\tau \geq 0$, $x \in \mathbb{R}^d$, $y \in \mathbb{R}^l$

\begin{equation*}
\left|\mathbf{E} \left [ \tau^{-1} \int_{t_0}^{t_0+\tau} b\left(x, y_{s}^{(x, y)}\right) d s-\overline{b}(x) \right] \right| \leq \overline{K}(\tau)\left(1+|x|^{2}+|y|^{2}\right)
\end{equation*}
 
\begin{equation*}
\left|\mathbf{E} \left [ \tau^{-1} \int_{t_0}^{t_0+\tau} a\left(x, y_{s}^{(x, y)}\right) d s-\overline{a}(x)\right] \right| \leq \overline{K}(\tau)\left(1+|x|^{2}+|y|^{2}\right)
\end{equation*}
 where $\lim_{\tau \to \infty} \overline{K}(\tau) = 0$ and $\{ y_s^{(x,y)} \}_{s \geq 0}$ satisfies the so called frozen equation
 \begin{align*}
 d y_{s}^{(x, y)}=B\left(x, y_{s}^{(x, y)}\right) d s+C\left(x, y_{s}^{(x, y)}\right) d \tilde{W}_{s} && y_{0}^{(x,y)}=y,s\geq 0
\end{align*}
In particular Khasminskii showed that the process $X^\epsilon$ converges weakly in $\mathcal{C}[0,T]^d$ to the solution $\overline{X}$ of the averaged equation
 \begin{align*} 
d\overline{X} _t = \overline{b}(\overline{X}_t)dt+\overline{\sigma}(\overline{X}_t)dW_t && \overline{X}_0=x_0, \forall t \in[0,T]
\end{align*}
The coefficients $\overline{b}$, $\overline{\sigma}$ satisfying the ergodic condition are usually averages of the initial ones $b$, $\sigma$ with respect to the invariant measure of the frozen equation which must be shown to exist. Typically a dissipativity condition implies the implicit hypotheses.
The proof of Khasminskii based on a time-discretization technique was very influential and it inspired many researches conducted by mathematicians like Freidlin, Wentzell, Veretennikov, Cerrai, Röckner and many others. We need to remark that when $\sigma$ is independent of the fast variable $Y^{\epsilon}$ we usually have the stronger convergence in probability of the slow variable $X^{\epsilon}$ (e.g. \cite{Cerrai}, \cite{friedlin}, \cite{veretnikov}). 

Recently non-autonomous slow-fast systems are being investigated:
In 2016 in \cite{Cerrai_lunardi} Cerrai and Lunardi studied a non-autonomous slow-fast system of reaction diffusion equations where the coefficients of the fast equation depend on time and satisfy the almost periodic in time condition. They proved the convergence in probability of $X^{\epsilon}$ when $\sigma$ is independent of $Y^{\epsilon}$.\\
More recently in 2020 in \cite{Wei_liu} the authors considered the non-autonomous system \eqref{eq:synopsis} with $\sigma$ independent of $Y^{\epsilon}$ and proved the $\mathcal{L}^p$-convergence of $X^{\epsilon}$ under local Lipschitz coefficients. Here the coefficients of the slow equation are assumed to depend also on $\omega$.\\
In \cite{wang} jumps were considered for slow-fast system of reaction diffusion equations. Also here the diffusion term of the fast equation is independent on the fast component.

In this work, motivated by the study of a local stochastic volatility model, we consider the non-autonomous slow-fast system \eqref{eq:synopsis}, where $\sigma$ depends on $Y^{\epsilon}$, and we prove the validity of the averaging principle. We assume the Lipschitzianity, an Holder's continuity in time and a sublinearity of the coefficients. Moreover in spirit of Khasminskii's work we assume an ergodic hypothesis and an $\mathcal{L}^2$-bound of the fast component.
In this setting we show the weak convergence of $X^{\epsilon}$ in $\mathcal{C}[0,T]^d$ to the solution of the averaged equation by means of a time discretization argument inspired by Khasminskii. The ergodic hypothesis and the $\mathcal{L}^2$-bound of the fast component assumed here are implicit hypotheses, so they need to be proved valid for a given slow-fast system. For this reason we also provide an explicit condition under which the implicit hypotheses are verified: this is the following dissipativity of the coefficients of the fast component
\begin{equation*}
<B(t,x,y_2)-B(t,x,y_1),y_2-y_1>+|C(t,x,y_2)-C(t,x,y_1)|^2 \leq -\beta |y_2-y_1|^2
\end{equation*}
We also treat the case of a perturbation of the drift of the fast equation by a slower term of the form $\epsilon^{-\eta}D(t,X_t^{\epsilon},Y_t^{\epsilon})$ for $0\leq \eta<1$ by considering
\begin{equation*} 
\begin{cases}
dX_t^{\epsilon}=b(t,X_t^{\epsilon},Y_t^{\epsilon})dt+\sigma(t,X_t^{\epsilon},Y_t^{\epsilon})dW_t \\
X_0^{\epsilon}=x_0 \\
dY_t^{\epsilon}=[\epsilon^{-1} B(t,X_t^{\epsilon},Y_t^{\epsilon})+\epsilon^{-\eta}D(t,X_t^{\epsilon},Y_t^{\epsilon})]dt+\epsilon^{-1/2}C(t,X_t^{\epsilon},Y_t^{\epsilon})d\tilde{W}_t \\
Y_0^{\epsilon}=y_0
\end{cases}
\end{equation*}
In fact this case will be necessary in the next part of the article.

At this point we propose a financial application of this result. In \cite{papanicolaou2} under the objective measure the authors consider a slow-fast stochastic volatility model of the form:
\begin{equation*}
\begin{cases}
dS_t^{\epsilon}=\eta S_t^\epsilon dt+ \mathcal{F}(Y_t^\epsilon)  S_t^\epsilon dW_t \\
S_0^{\epsilon}=s_0 \\
dY_t^{\epsilon}=\epsilon^{-1} \mathcal{B}(Y_t^{\epsilon})dt+\epsilon^{-1/2} \mathcal{C}(Y_t^{\epsilon})  d\tilde{W}_t \\
Y_0^{\epsilon}=y_0 
\end{cases}
\end{equation*}
where $S^\epsilon$ is the stock price, $\mathcal{F}(y)$ is the stochastic volatility driven by an ergodic, mean-reverting process $Y^\epsilon$.\\
Then, after the change of measure to a risk neutral one, they perform a perturbational analysis of the prices of the financial derivatives. This analysis is performed through an asymptotic expansion as $\epsilon \to 0$ of the pricing partial differential equation. An exhaustive presentation would be impossible here, we refer to \cite{papanicolaou2} for all the details. The literature on these models and their generalizations is very rich, e.g. see \cite{choi_hybrid}, \cite{papanicolaou2}, \cite{papanicolaou}, \cite{papanicolaou3}, \cite{lorig_lsv} and references therein.

Inspired by these works under the objective measure we consider a slow-fast local stochastic volatility model of the form
\begin{equation*}
\begin{cases}
dS_t^{\epsilon}=\mathcal{H}(t,S_t^{\epsilon},Y_t^\epsilon)  S_t^\epsilon dt+ \mathcal{F}(t,S_t^{\epsilon},Y_t^\epsilon)  S_t^\epsilon dW_t \\
S_0^{\epsilon}=s_0 \\
dY_t^{\epsilon}=\epsilon^{-1} \mathcal{B}(t,S_t^{\epsilon},Y_t^\epsilon) dt+\epsilon^{-1/2} \mathcal{C}(t,S_t^{\epsilon},Y_t^\epsilon)   d\tilde{W}_t \\
Y_0^{\epsilon}=y_0 
\end{cases}
\end{equation*}
where $S^\epsilon$ is the stock price, $\mathcal{H}(t,s,y)$, $\mathcal{F}(t,s,y)$ are the local stochastic drift and local stochastic volatility respectively and they are driven by the fast process $Y^\epsilon$. For local stochastic volatility models refer to \cite{lorig2} and \cite{lorig_lsv}. See also \cite{bergomi} for a general overview. Applying the theory just developed, we show the weak convergence of $S^\epsilon$ in $\mathcal{C}[0,T]$ to the solution $\overline{S}$ of the averaged equation 
\begin{align*}
d\overline{S} _t =  \overline{\mathcal{H}}(t,\overline{S} _t) \overline{S}_t  dt+\overline{\mathcal{F}}(t,\overline{S} _t) \overline{S}_t dW_t && \overline{S}_0=s_0, \forall t \in[0,T]
\end{align*}
which turns out to be a local volatility model.

At this point we proceed with the construction of a family of risk neutral measures, under which pricing can be done, and we consider a financial derivative possibly path-dependent. Then, under a risk neutral measure, we show the weak convergence of $S^\epsilon$ in $\mathcal{C}[0,T]$ to $\overline{S}$. Moreover we prove that the price of the derivative, calculated using the slow-fast model $(S^\epsilon, Y^\epsilon)$, converge to the price of the derivative calculated using the limit model $\overline{S}$. 

The work is organized as follows:\\
in section \ref{sec:The Averaging Principle for a Class of non Autonomous Stochastic Differential Equations} we introduce a class of non-autonomous slow-fast stochastic differential equations and we prove that the averaging principle holds for such class.\\
In section \ref{sec:Explicit conditions} we prove that the implicit hypotheses assumed in section \ref{sec:The Averaging Principle for a Class of non Autonomous Stochastic Differential Equations} are verified under an explicit dissipativity condition.\\
In section \ref{sec:Financial Application} we apply the result shown in the previous chapters to a financial local stochastic volatility model.

\section{The Averaging Principle for a Class of non Autonomous Stochastic Differential Equations}
\label{sec:The Averaging Principle for a Class of non Autonomous Stochastic Differential Equations}
\subsection{Framework and main Theorem}
Let $(\Omega, \mathcal{F}, \mathcal{F}_t,\left(W_t,Z_t \right) , \mathbb{P})_{t \geq 0}$ be a continuous standard $\mathbb{R}^{d+l}$-Brownian motion.\\
As in \cite[section 6]{bardi} we construct an $\mathbb{R}^{l}$-Brownian motion $\{ \tilde{W}_t \}_{t \geq 0}$ correlated to $\{ W_t \}_{t \geq 0}$
by defining 
\begin{align}
    \tilde{W}_t^j=\sum_{i=1}^d \rho_{ij}W_t^i +  \left (1- \sum_{i=1}^d \rho_{ij}^2 \right)^{1/2} Z_t^j && \forall j \leq l
\end{align}
where $\rho_{ij} \in (-1,1)$, $\sum_{i=1}^d \rho_{ij}^2 \leq 1$ for every $j \leq l$ and $\sum_{i=1}^d \rho_{ij}\rho_{ik}=0$ for every $j \neq k$.\\
In this way we have the following correlation:
\begin{align*}
\mathbb{E}[dW_t^i d\tilde{W}_t^j]=\rho_{ij} dt  && \forall i \leq d, j \leq l
\end{align*}
Let's now consider the following system of stochastic differential equations:
\begin{equation} \label{eq:system}
\begin{cases}
dX_t^{\epsilon}=b(t,X_t^{\epsilon},Y_t^{\epsilon})dt+\sigma(t,X_t^{\epsilon},Y_t^{\epsilon})dW_t \\
X_0^{\epsilon}=x_0 \\
dY_t^{\epsilon}=\epsilon^{-1} B(t,X_t^{\epsilon},Y_t^{\epsilon})dt+\epsilon^{-1/2}C(t,X_t^{\epsilon},Y_t^{\epsilon})d\tilde{W}_t \\
Y_0^{\epsilon}=y_0
\end{cases}
\end{equation}
where $\epsilon \in (0,1]$, t $\in [0,T]$, $T$ is the time horizon , $x_0  \in \mathbb{R}^d$, $y_0 \in \mathbb{R}^l$, $b$, $\sigma$, $B$, $C$ are Borel functions:
\begin{align*}
b \colon & [0,T] \times \mathbb{R}^d \times \mathbb{R}^l  \longrightarrow \mathbb{R}^d \\
B \colon  &[0,T] \times \mathbb{R}^d \times \mathbb{R}^l  \longrightarrow \mathbb{R}^l \\
\sigma\colon  &[0,T] \times \mathbb{R}^d \times \mathbb{R}^l  \longrightarrow \mathbb{R}^d  \times \mathbb{R}^d \\
C\colon  &[0,T] \times \mathbb{R}^d \times \mathbb{R}^l  \longrightarrow \mathbb{R}^l  \times \mathbb{R}^l 
\end{align*}
Now we list the following hypothesis under which we will show the main theorem of this work:
\begin{hypothesis} \label{hp:lip}
There exist $L>0$, $\gamma>0$ such that:
 \begin{align*} 
&\abs{b(t,x_2,y_2)-b(s,x_1,y_1)}  \leq L (|t-s|^\gamma+|x_2-x_1|+|y_2-y_1|) \\
&\abs{\sigma(t,x_2,y_2)-\sigma(s,x_1,y_1)}  \leq L (|t-s|^\gamma+|x_2-x_1|+|y_2-y_1|) \\
&\abs{B(t,x_2,y_2)-B(s,x_1,y_1)}  \leq L (|t-s|^\gamma+|x_2-x_1|+|y_2-y_1|) \\
&\abs{C(t,x_2,y_2)-C(s,x_1,y_1)}  \leq L (|t-s|^\gamma+|x_2-x_1|+|y_2-y_1|)
\end{align*}
for every $0 \leq t,s \leq T$, $x_1,x_2 \in \mathbb{R}^d$, $y_1,y_2 \in \mathbb{R}^l$.
\end{hypothesis}
\begin{hypothesis} \label{hp:sub} There exists $M>0$ such that:
\begin{align*} 
&\abs{b(t,x,y)}  \leq M (1+\abs{x})  \\
&\abs{\sigma(t,x,y)}  \leq M (1+\abs{x}) \\
&\abs{B(t,x,y)} \leq M (1+\abs{x}+\abs{y}) \\
&\abs{C(t,x,y)}  \leq M (1+\abs{x}+\abs{y})
\end{align*}
for every $0 \leq t \leq T$, $x \in \mathbb{R}^d$, $y \in \mathbb{R}^l$.
\end{hypothesis}
\begin{remark}
Analogously to what happens in \cite{Khasminskii} we can't assume the full sublinearity of $b$ and $\sigma$ because the behaviour $Y^\epsilon$ is too fast.
\end{remark}

\begin{proposition}
Under hypothesis \ref{hp:lip}, \ref{hp:sub} equation \eqref{eq:system} has a unique strong solution $\{ (X_t^{\epsilon},Y_t^{\epsilon}) \}_{t \in [0,T]}$ for every $ \epsilon >0$.
\end{proposition}
\begin{proof}
It follows by the classic theorem of existence and uniqueness of strong solutions after noticing that we can write system  \eqref{eq:system} with respect to the independent Brownian motions $W_t,Z_t$ by means of a linear transformation.
\end{proof}
Now we define the square d-dimensional matrix $a(t,x,y)=\sigma (t,x,y) \sigma^T(t,x,y)/2$ and we state the ergodic hypothesis.
\begin{hypothesis} \label{hp:est} There exist a continuous d-dimensional vector $\overline{b}(t,x)$ and a continuous square d-dimensional matrix $\overline{\sigma} (t,x)$ such that for every $ 0 \leq t\leq T$, $t_0 \geq 0$, $\tau \geq 0$, $x \in \mathbb{R}^d$, $y \in \mathbb{R}^l$

\begin{equation}\label{eq:erg1}
\left|\mathbf{E} \left [ \tau^{-1} \int_{t_0}^{t_0+\tau} b\left(t,x, y_{s}^{t,x, y}\right) d s-\overline{b}(t,x) \right] \right| \leq \overline{K}(\tau)\left(1+|x|^{2}+|y|^{2}\right)
\end{equation}
 
\begin{equation}\label{eq:erg2}
\left|\mathbf{E} \left [ \tau^{-1} \int_{t_0}^{t_0+\tau} a\left(t,x, y_{s}^{t,x, y}\right) d s-\overline{\sigma} (t,x) \overline{\sigma}^T(t,x)/2\right] \right| \leq \overline{K}(\tau)\left(1+|x|^{2}+|y|^{2}\right)
\end{equation}
 where 
 $$\lim_{\tau \to \infty} \overline{K}(\tau) = 0$$
 and $\{ y_s^{t,x,y}\}_{s \geq 0}$ is the unique strong solution to the frozen equation:  
 
\begin{align}\label{eq:frozen}
 d y_{s}^{t,x, y}=B\left(t,x, y_{s}^{t,x, y}\right) d s+C\left(t,x, y_{s}^{t,x, y}\right) d \tilde{W}_{s} && y_{0}^{t,x,y}=y
\end{align}
Moreover we assume that $\overline{b}(t,x), \overline{\sigma}(t,x)$ are locally Lipschitz in $x$ uniformly with respect to $ t \leq T$.
\end{hypothesis}
\begin{flushleft}
Now we define the matrix $\overline{ a}(t,x)=\overline{\sigma} (t,x) \overline{\sigma}^T(t,x)/2$.
\end{flushleft}

\begin{remark}
Hypothesis \ref{hp:est} is the generalization to the non autonomous case of Khasminskii's ergodic hypothesis in \cite{Khasminskii}. 
\end{remark}

\begin{remark}
It is not necessary to assume that $\overline{b}$ and $\overline{\sigma}$ are continuous in $(t,x)$ since it is possible to prove it by using \eqref{eq:erg1} and \eqref{eq:erg2}. On the other hand in order to have a unique strong solution of the averaged equation we must assume that they are locally Lipschitz as it doesn't follow from \eqref{eq:erg1} and \eqref{eq:erg2}. 
\end{remark}

\begin{remark}
$\{ y_s^{t,x,y}\}_{s \geq 0}$ is the time changed process $s \to \epsilon s$ with the time $t$-component and the $x$-component frozen. Moreover the frozen equation has a unique strong solution thanks to hypothesis \ref{hp:lip}, \ref{hp:sub}.
\end{remark}

\begin{remark}
Letting $\tau \to +\infty$ in \eqref{eq:erg1}, \eqref{eq:erg2} we obtain the following characterizations:
\begin{align}\label{eq:charact_b}
\overline{b}(t,x)=\lim_{\tau \to +\infty} \mathbf{E} \left [\tau^{-1} \int_{t_0}^{t_0+\tau} b\left(t,x, y_{s}^{t,x, y}\right) d s \right] && \forall y \in \mathbb{R}^l,t_0 \geq 0
\end{align}
\begin{align}\label{eq:charact_sigma}
\overline{a}(t,x)=\lim_{\tau \to +\infty} \mathbf{E} \left [\tau^{-1} \int_{t_0}^{t_0+\tau} a\left(t,x, y_{s}^{t,x, y}\right) d s \right] &&     \forall y \in \mathbb{R}^l,t_0 \geq 0
\end{align}
\end{remark}
\begin{remark}
For what concerns $\overline{\sigma} (t,x)$ usually one calculates first the matrix $\overline{a}(t,x)$ as the ergodic limit \eqref{eq:charact_sigma} and proves that \eqref{eq:erg2} holds. Then  the matrix $\overline{\sigma} (t,x)$ can be chosen in the following way: since $a(t,x,y)$ is symmetric and positive semidefinite this implies that $\overline{a}(t,x)$ is also symmetric and positive semidefinite and so it has a unique symmetric positive semidefinite square root $\overline{\sigma} (t,x)=\sqrt{2\overline{a}(t,x)}$. 
\end{remark}
\begin{remark}
Let's observe that the matrix $\overline{\sigma} (t,x)$ satisfying \eqref{eq:erg2} could be not unique.   In fact, given the symmetric positive semidefinite matrix $\overline{a}(t,x)$, the matrix $\overline{\sigma} (t,x)$ such that $\overline{ a}(t,x)=\overline{\sigma} (t,x) \overline{\sigma}^T(t,x)/2$ in general is not unique. This means that we would end up with different averaged equations of the form \eqref{eq:main}. Anyway this is not a problem because the law of the solution of \eqref{eq:main} depends on $\overline{\sigma} (t,x)$ only through $\overline{a}(t,x)$, e.g. see \cite[chapter 5]{strook}.
\end{remark}
\begin{flushleft}
Let now $B_B(\mathbb{R}^l)$ be the space of bounded measurable functions.
\end{flushleft}

\begin{remark}
\label{rem:candidates}
If for fixed $t, x$ the transition semigroup of the frozen equation $ \{ \mathcal{P}^{t,x}_s \}_{s \geq 0}$ defined by 
\begin{align*}
\mathcal{P}^{t,x}_s\phi(y)=\mathbb{E}[\phi(y_s^{t,x,y})] && \phi \in B_B(\mathbb{R}^l),y \in \mathbb{R}^l
\end{align*}
has a unique invariant measure $\mu^{t,x}$ with a certain rate of decay to the equilibrium, then the vector $$\int_{\mathbb{R}^l}b(t,x,y)\mu^{t,x}(dy)$$ and the matrix $$\int_{\mathbb{R}^l}a(t,x,y)\mu^{t,x}(dy)$$ are the natural candidates for $\overline{b}(t,x)$ and $\overline{a}(t,x)$ respectively.
\end{remark}
\begin{remark} \label{rem:sublinearity_bar}
Thanks to \eqref{eq:charact_b}, \eqref{eq:charact_sigma} $\overline{b}$, $\overline{\sigma}$ satisfy
\begin{align*} 
\abs{\overline{b}(t,x)} & \leq M (1+\abs{x})  \\
\abs{\overline{\sigma}(t,x)} & \leq M (1+\abs{x})
\end{align*}
for every $0 \leq t \leq T$, $x \in \mathbb{R}^d$.
\end{remark}
Now we assume an $\mathcal{L}^2$-bound of the fast component:
\begin{hypothesis} \label{hp:expectation} 
There exists $C>0$ such that
\begin{equation*}
\sup_{\epsilon \in (0,1]} \sup_{0 \leq t \leq T} \mathbb{E}\left[ \abs{Y_t^\epsilon}^2 \right] \leq C
\end{equation*}
\end{hypothesis}
\begin{remark}
Hypotheses \ref{hp:est}, \ref{hp:expectation} are implicit conditions: hence they need to be proven valid for a given slow-fast system of stochastic differential equations. In the next section we will provide an explicit condition under which they are satisfied, following the idea of remark \ref{rem:candidates}. 
\end{remark}

We can now introduce the main theorem of this work which states that the averaging principle holds for system \eqref{eq:system}.
\begin{theorem}[Averaging Principle] \label{th:averaging}
Assume that hypotheses \ref{hp:lip}, \ref{hp:sub}, \ref{hp:est}, \ref{hp:expectation} hold. Then  $ X^{\epsilon} \to \overline{X}$ weakly in  $\mathcal{C}[0,T]^d$ as $\epsilon \to 0$ where the process $\overline{X}$ is the unique strong solution of the averaged equation
\begin{align} \label{eq:main}
d\overline{X} _t = \overline{b}(t,\overline{X}_t)dt+\overline{\sigma}(t,\overline{X}_t)dW_t && \overline{X}_0=x_0, \forall t \in[0,T]
\end{align}
\end{theorem} 
\begin{flushleft}
Here we have considered the following definition:
\end{flushleft}
\begin{definition}
Let $\{ (\Omega_\epsilon,\mathcal{F}_\epsilon,P_\epsilon) \}_{\epsilon>0}$ be a family of probability spaces. On each of them consider a random variable $X_\epsilon$ with values in a complete metric space $S$. Let $(\Omega,\mathfrak{F},P)$ be another probability space on which a random variable $X$ with values in $S$ is given. We say that $X_\epsilon \to X$ weakly in $S$ as $\epsilon \to 0$ if the family of probability measures induced by $\{X_\epsilon\}_{\epsilon>0 }$ converges weakly as $\epsilon \to 0$ to the probability measure induced by $X$.
\end{definition}
\begin{remark}
We remind that $\mathcal{C}[0,T]^d$ is a complete metric space with metric
\begin{align*}
\rho(y_2,y_1) = \sup_{t \in[0,T] } \abs{y_2(t)-y_1(t)} &&  y_1,y_2 \in \mathcal{C}[0,T]^d
\end{align*}
\end{remark}
\begin{remark}\label{rem:uniqueness_averaged_eq}
The averaged equation \eqref{eq:main} has a unique strong solution thanks to hypothesis \ref{hp:est} and remark \ref{rem:sublinearity_bar} .
\end{remark}
\subsection{Proof of the Main Theorem}
We start by constructing an auxiliary process obtained by a discretization in time: fix $\Delta \in [0,T]$ and let $K=0...\lceil T/ \Delta \rceil-1$ 
\begin{equation}\label{eq:discretization}
\begin{cases}
d\hat {X}_t^{\epsilon}=b(K \Delta, X_{K \Delta}^{\epsilon},\hat{Y}_t^{\epsilon})dt+\sigma(K \Delta,X_{K \Delta}^{\epsilon},\hat{Y}_t^{\epsilon})dW_t & t \in [K \Delta, ((K+1) \Delta) \wedge T )\\
\hat {X}_{0}^{\epsilon}=x_0 \\
d\hat{Y}_t^{\epsilon}=\epsilon^{-1} B(K \Delta,X_{K \Delta}^{\epsilon},\hat{Y}_t^{\epsilon})dt+\epsilon^{-1/2}C(K \Delta,X_{K \Delta}^{\epsilon},\hat{Y}_t^{\epsilon})d\tilde{W}_t & t \in [K \Delta, ((K+1) \Delta) \wedge T) \\
\hat {Y}_{K \Delta}^{\epsilon}=Y_{K \Delta}^{\epsilon} 
\end{cases}
\end{equation}
and we choose the length of the intervals $\Delta=\Delta(\epsilon) = \epsilon (\log {\epsilon^{-1}})^{1/4}$ such that as $\epsilon \to 0$ then $\Delta\to 0$ and $\Delta / \epsilon \to + \infty$.
\begin{remark}
Following Khasminskii's idea the proof is based on the previous discretization. We notice that in each time interval the time component and the slow component of the coefficients are fixed while the fast component is free to move. In fact this is the idea of the proof: to consider an approximated process in which only the fast component evolves because it is much faster than the other two $t$ and $x$ which can then be considered fixed.
\end{remark}
We now prove the following a priori estimates:
\begin{lemma} 
Let $p \geq 2$, then there exists a constant $C=C(p,T)>0$ such that $\forall \epsilon >0 $
\begin{equation}
\mathbb{E}\left[\sup_{t \in [0,T]} |X_t^{\epsilon}|^p \right] \leq C(1+|x_0|^p)  \label{diseq1:lemma}
\end{equation}
\begin{equation}
\mathbb{E}\left[\sup_{t \in [0,T]} |\hat{X}_t^{\epsilon}|^p \right] \leq C(1+|x_0|^p) \label{diseq2:lemma}
\end{equation}

\end{lemma}

\begin{proof}
Let's fix $R>0$ and consider the stopping time $\tau_R=\inf (t>0, |X_t^{\epsilon}| \geq R)$. \\
Let's define $X_R^{\epsilon}(t)=X_{\tau_R \land t }^{\epsilon}$, then:

\begin{equation*}
X_R^{\epsilon}(t)=x_0+\int_{0}^{t} b(r,X_R^{\epsilon}(r),Y_r^{\epsilon}) \mathbbm{1}_{r \leq \tau_R} dr+\int_{0}^{t} \sigma(r,X_R^{\epsilon}(r),Y_r^\epsilon) \mathbbm{1}_{r \leq \tau_R} dW_r
\end{equation*}
It follows:

\begin{align*} 
\mathbb{E}\left [\sup_{s \in [0,t]} |X_R^{\epsilon}(s)|^p \right] & \leq 3 ^{p-1}|x_0|^p + 3 ^{p-1} \mathbb{E} \left [\sup_{s \in [0,t]} \abs{\int_{0}^{s} b(r,X_R^{\epsilon}(r),Y_r^{\epsilon}) \mathbbm{1}_{r \leq \tau_R} dr}^p \right]  \\ 
& + 3 ^{p-1} \mathbb{E} \left [\sup_{s \in [0,t]} \abs{\int_{0}^{s} \sigma(r,X_R^{\epsilon}(r),Y_r^{\epsilon}) \mathbbm{1}_{r \leq \tau_R} dW_r}^p \right ] 
\end{align*}
By Hölder inequality and by sublinearity we obtain the estimate for the second term on the right hand side:

\begin{align*}
\mathbb{E} \left [\sup_{s \in [0,t]} \abs{\int_{0}^{s} b(r,X_R^{\epsilon}(r),Y_r^{\epsilon}) \mathbbm{1}_{r \leq \tau_R} dr}^p \right]
& \leq t^{p-1} \mathbb{E} \left [ \int_{0}^{t} \abs {b(r,X_R^{\epsilon}(r),Y_r^{\epsilon}) \mathbbm{1}_{r \leq \tau_R} }^p dr \right] \\
& \leq T^{p-1} M^p \mathbb{E} \left [ \int_{0}^{t} (1+|X_R^{\epsilon}(r)|)^p dr \right] 
\end{align*}
By Burkholder-Davis-Gundy and Hölder inequalities and by sublinearity we have the estimate for the third term on the right hand side:
\begin{align*}
\mathbb{E} \left [\sup_{s \in [0,t]} \abs{\int_{0}^{s} \sigma(r,X_R^{\epsilon}(r),Y_r^{\epsilon}) \mathbbm{1}_{r \leq \tau_R} dW_r}^p \right ] & \leq C_p t^{(p-2)/2} \mathbb{E} \left [ \int_{0}^{t} \abs {\sigma(r,X_R^{\epsilon}(r),Y_r^{\epsilon}) }^p dr \right] \\
& \leq C_p T^{(p-2)/2} M^p \mathbb{E} \left [ \int_{0}^{t} (1+|X_R^{\epsilon}(r)|)^p dr \right] 
\end{align*}
Putting everything together we obtain:

\begin{align*}
\mathbb{E}\left [\sup_{s \in [0,t]} |X_R^{\epsilon}(s)|^p \right] & \leq 3^{p-1} x_0 + C_1 \mathbb{E} \left [ \int_{0}^{t} (1+|X_R^{\epsilon}(r)|)^p dr \right]  \notag \\ 
& \leq C_2 (1+|x_0|^p)+C_3 \int_0^t \mathbb{E} \left [\sup_{s \in [0,r]}|X_R^{\epsilon}(s)|^p dr \right ]
\end{align*}
where the constants $C_1,C_2,C_3$ are independent of $\epsilon$.\\
Defining $$v(t)=\mathbb{E}\left [\sup_{s \in [0,t]} |X_R^{\epsilon}(s)|^p \right] $$ 
we can apply Gronwall's Lemma because $v(t) \leq \max (R^p, |x_0|^p) < + \infty $ obtaining:

\begin{equation}\label{eq:thesis_lemma1}
v(T)=\mathbb{E}\left [\sup_{s \in [0,T]} |X_R^{\epsilon}(s)|^p \right] \leq C_2(1+|x_0|^p) e^{T C_3}=C(1+|x_0|^p) 
\end{equation}
where the constant $C$ is independent of $\epsilon$.\\
Letting $R \rightarrow \infty$ we see that $\tau_R \rightarrow + \infty$ a.s. with respect to $\mathbb{P}$ and this implies:
\begin{align*}
\lim_{R\to + \infty} \sup_{s \in [0,T]} |X_R^{\epsilon}(s)|^p = \sup_{s \in [0,T]} |X^{\epsilon}(s)|^p && a.s.
\end{align*}
Applying Beppo Levi's theorem to~\eqref{eq:thesis_lemma1} we get \eqref{diseq1:lemma}.\\
Using~\eqref{diseq1:lemma} it's easy to prove~\eqref{diseq2:lemma}.
\end{proof}
We now prove the following lemma:

\begin{lemma} \label{lemma:lemma}
Let $p \geq 2$, then there exists a constant $C=C(p,T)>0$ such that $\forall \epsilon >0$
\begin{equation} \label{eq:lemma_delta1}
\mathbb{E}\left[ |\hat{X}_{t_2}^{\epsilon}-\hat{X}_{t_1}^{\epsilon}|^p \right] \leq C |t_2-t_1|^{p/2}
\end{equation}
\begin{align} \label{eq:lemma_delta2}
\mathbb{E}\left[|X_{t_2}^{\epsilon}-X_{t_1}^{\epsilon}|^p \right] \leq C |t_2-t_1|^{p/2}
\end{align}
for every  $ t_1,t_2 \in [0,T]$.
\end{lemma}

\begin{proof}
Let $t_1 \leq t_2 \in [0,T]$, by usual calculation we have:
\begin{align*}
\mathbb{E}[ |\hat{X}_{t_2}^{\epsilon}-\hat{X}_{t_1}^{\epsilon}|^p] &
\leq 2^{p-1} \mathbb{E} \left [\abs{ \int_{t_1}^{t_2} b(\lfloor s/ \Delta \rfloor \Delta,X_{(\lfloor s/ \Delta \rfloor \Delta)}^{\epsilon},\hat{Y}_s^{\epsilon})ds}^p \right]\\
& + 2^{p-1}  \mathbb{E} \left [\abs{ \int_{t_1}^{t_2} \sigma(\lfloor s/ \Delta \rfloor \Delta,X_{(\lfloor s/ \Delta \rfloor \Delta)}^{\epsilon},\hat{Y}_s^{\epsilon})dW_s}^p  \right ] \\ &
\leq 2^{p-1} \mathbb{E} \left [ \abs{t_2-t_1 }^{p-1} \int_{t_1}^{t_2} \abs{ b(\lfloor s/ \Delta \rfloor \Delta,X_{(\lfloor s/ \Delta \rfloor \Delta)}^{\epsilon},\hat{Y}_s^{\epsilon})}^{p} ds  \right ] \\ &
+2^{p-1} C_p \mathbb{E} \left [ \abs{t_2-t_1}^{(p-2)/2} \int_{t_1}^{t_2} \abs{ \sigma (\lfloor s/ \Delta \rfloor \Delta,X_{(\lfloor s/ \Delta \rfloor \Delta)}^{\epsilon},\hat{Y}_s^{\epsilon})}^{p} ds  \right ] \\ &
\leq 2^{p-1} \abs{t_2-t_1 }^{p-1} \mathbb{E} \left [ \int_{t_1}^{t_2} M^p \left (1+\abs{X_{(\lfloor s/ \Delta \rfloor \Delta)}^{\epsilon}} \right )^p ds  \right ] \\ &
+2^{p-1} C_p \abs{t_2-t_1 }^{(p-2)/2} \mathbb{E} \left [ \int_{t_1}^{t_2} M^p \left (1+\abs{X_{(\lfloor s/ \Delta \rfloor \Delta)}^{\epsilon}} \right )^p ds  \right ] 
\end{align*}
where $C_p$ is independent of $\epsilon$.\\
Finally using~\eqref{diseq1:lemma} we get~\eqref{eq:lemma_delta1}, in fact:
\begin{align*}
\mathbb{E}[|\hat{X}_{t_2}^{\epsilon}-\hat{X}_{t_1}^{\epsilon}|^p] & \leq C_1 \abs{t_2-t_1}^{p-1} \abs{t_2-t_1} +C_2  \abs{t_2-t_1}^{(p-2)/2} \abs{t_2-t_1}\\ & \leq C \abs{t_2-t_1}^{p/2}
\end{align*}
where the constants $C,C_1,C_2$ are independent of  $\epsilon$.\\
In the same way it's possible to prove~\eqref{eq:lemma_delta2}.
\end{proof}
We are now ready to prove theorem~\ref{th:averaging}.

\begin{proof}
We start by showing the following relation:
\begin{equation}  \label{eq2_th}
\lim_{\epsilon \to 0} \mathbb{E} \left [\sup_{t \in [0,T]} |X_t^{\epsilon}-\hat{X}_t^{\epsilon}|^2 \right] =0
\end{equation}
In order to do this we first prove:
\begin{equation}  \label{eq1_th}
\lim_{\epsilon \to 0} \sup_{x_0,y_0} \sup_{t \in [0,T]} \mathbb{E} \left [|Y_t^{\epsilon}-\hat{Y}_t^{\epsilon}|^2 \right] =0
\end{equation}
Let's fix $t \in [K \Delta, ((K+1) \Delta) \wedge T)$, then:
\begin{align*}
\mathbb{E}[|Y_t^{\epsilon}-\hat{Y}_t^{\epsilon}|^2] &
\leq 2 \mathbb{E} \left [ \epsilon^{-2} \abs{ \int_{K \Delta}^t B(s,X_s^\epsilon,Y_s^\epsilon)-B(K\Delta,X_{K \Delta}^{\epsilon},\hat{Y}_s^{\epsilon})ds}^2 \right] \\
& +2 \mathbb{E} \left [ \epsilon^{-1} \abs{ \int_{K \Delta}^t C(s,X_s^\epsilon,Y_s^\epsilon)-C(K\Delta, X_{K \Delta}^{\epsilon},\hat{Y}_s^{\epsilon})d \tilde{W}_s}^2  \right ] \\
& = 2(I_1^\epsilon(t) + I_2^\epsilon(t))
\end{align*}
First we consider the term $I_1^\epsilon(t)$. By Hölder inequality and hypothesis \ref{hp:lip} we have:
\begin{align*}
I_1^\epsilon(t) & \leq  \epsilon^{-2} \Delta \mathbb{E}  \Bigg[  \int_{K \Delta}^t \abs{B(s,X_s^\epsilon,Y_s^\epsilon)-B(K\Delta,X_{K \Delta}^{\epsilon},\hat{Y}_s^{\epsilon})}^2ds \Bigg] \\
& \leq 3 \epsilon^{-2} \Delta \mathbb{E} \Bigg [ \int_{K \Delta}^t \abs{ B(s,X_s^\epsilon,Y_s^\epsilon)-B(K\Delta,X_s^\epsilon,Y_s^{\epsilon})}^2 ds\Bigg]\\
& +3 \epsilon^{-2}\Delta  \mathbb{E}  \Bigg[ \int_{K \Delta}^t \abs{B(K\Delta,X_s^\epsilon,Y_s^{\epsilon})- B(K\Delta,X_{K \Delta}^\epsilon,Y_s^\epsilon)}^2 ds\Bigg]\\
&  +3 \epsilon^{-2}\Delta  \mathbb{E}  \Bigg[\int_{K \Delta}^t \abs{B(K\Delta,X_{K \Delta}^\epsilon,Y_s^\epsilon)-B(K\Delta,X_{K \Delta}^{\epsilon},\hat{Y}_s^{\epsilon})}^2 ds  \Bigg] \\
& \leq 3 L^2 \Delta \epsilon^{-2}  \int_{K \Delta}^t (s-K\Delta)^{2\gamma}+ \mathbb{E} \left [ \abs{ X_s^\epsilon - X_{K \Delta}^\epsilon }^2 \right ] + \mathbb{E} \left [ \abs{Y_s^\epsilon - \hat{Y}_s^{\epsilon} }^2 \right] ds
\end{align*}
An analogous estimate can be obtained for the term $I_2^\epsilon(t)$ using Ito's Isometry and hypothesis \ref{hp:lip}:
\begin{align*}
I_2^\epsilon(t) & =  \epsilon^{-1} \mathbb{E}  \Bigg[  \int_{K \Delta}^t \abs{C(s,X_s^\epsilon,Y_s^\epsilon)-C(K\Delta,X_{K \Delta}^{\epsilon},\hat{Y}_s^{\epsilon}) }^2ds \Bigg] \\
& \leq 3 L^2  \epsilon^{-1}  \int_{K \Delta}^t (s-K\Delta)^{2\gamma}+ \mathbb{E} \left [ \abs{ X_s^\epsilon - X_{K \Delta}^\epsilon }^2 \right ] + \mathbb{E} \left [ \abs{Y_s^\epsilon - \hat{Y}_s^{\epsilon} }^2 \right] ds
\end{align*}
Putting everything together we have:
\begin{equation*}
\mathbb{E}[|Y_t^{\epsilon}-\hat{Y}_t^{\epsilon}|^2] \leq  6 L^2(\Delta \epsilon^{-2} + \epsilon^{-1}) \int_{K \Delta}^t (s-K\Delta)^{2\gamma} + \mathbb{E} \left [ \abs{ X_s^\epsilon - X_{K \Delta}^\epsilon }^2 \right ] + \mathbb{E} \left [ \abs{Y_s^\epsilon - \hat{Y}_s^{\epsilon} }^2 \right] ds
\end{equation*}
By lemma \ref{lemma:lemma} and since $t \in [K \Delta, ((K+1) \Delta) \wedge T)$, we have:
\begin{align*}
\mathbb{E}[|Y_t^{\epsilon}-\hat{Y}_t^{\epsilon}|^2] &  \leq 6 L^2(\Delta \epsilon^{-2} + \epsilon^{-1}) \int_{K \Delta}^t \Delta^{2\gamma} +C \Delta+ \mathbb{E} \left [ \abs{Y_s^\epsilon - \hat{Y}_s^{\epsilon} }^2 \right] ds \\
& \leq 6 L^2 (\Delta \epsilon^{-2} + \epsilon^{-1})(\Delta^{1+2\gamma} +C \Delta^2)  + 6L^2 (\Delta \epsilon^{-2} + \epsilon^{-1}) \int_{K \Delta}^t \mathbb{E} \left [ \abs{Y_s^\epsilon - \hat{Y}_s^{\epsilon} }^2 \right] ds 
\end{align*}
By Gronwall's lemma it follows:
\begin{equation}\label{eq_estimate_gronwall}
\mathbb{E}[|Y_t^{\epsilon}-\hat{Y}_t^{\epsilon}|^2] \leq 6 L^2  (\Delta \epsilon^{-2} + \epsilon^{-1}) (\Delta^{1+2\gamma} +C \Delta^2) \exp{ (6L^2 (\Delta^2 \epsilon^{-2} + \Delta \epsilon^{-1}))} 
\end{equation}
which tends to zero as as $\epsilon \to 0$ since $\Delta = \epsilon (\log {\epsilon^{-1}})^{1/4}$. This implies~\eqref{eq1_th}.\\
We can now prove~\eqref{eq2_th}:
\begin{align*}
 \mathbb{E} \left [\sup_{t \in [0,T]} |X_t^{\epsilon}-\hat{X}_t^{\epsilon}|^2 \right] 
 & \leq 2 \mathbb{E} \left [ \sup_{t \in [0,T]} \abs { \sum_{K=0}^{\lceil t/ \Delta \rceil -1} \int_{K \Delta}^{[(K+1) \Delta]\wedge t}b(K\Delta,X_{K \Delta}^\epsilon,\hat{Y}_s^\epsilon)-b(s,X_{s}^{\epsilon},Y_s^{\epsilon} ) ds}^2  \right ] \\ 
 & + 2 \mathbb{E} \left [\sup_{t \in [0,T]}  \abs{ \sum_{K=0}^{\lceil t/ \Delta \rceil -1} \int_{K \Delta}^{[(K+1) \Delta]\wedge t}\sigma (K\Delta,X_{K \Delta}^\epsilon,\hat{Y}_s^\epsilon)-\sigma (s,X_{s}^{\epsilon},Y_s^{\epsilon}) dW_s}^2  \right ] \\
 & = 2(E_1^\epsilon+E_2^\epsilon)
\end{align*}
We consider the first term on the right hand side $E_1^\epsilon$.\\ By Holder's inequality and hypothesis \ref{hp:lip} we have:
\begin{align*}
E_1^\epsilon 
& \leq \Delta  \mathbb{E} \left [\sum_{K=0}^{\lceil T/ \Delta \rceil -1} \int_{K \Delta}^{[(K+1) \Delta]\wedge T} \abs { b(K\Delta,X_{K \Delta}^\epsilon,\hat{Y}_s^\epsilon)-b(s,X_{s}^{\epsilon},Y_s^{\epsilon} ) }^2 ds  \right ]\\
& \leq 3 \Delta \mathbb{E} \left [ \sum_{K=0}^{\lceil T/ \Delta \rceil -1} \int_{K \Delta}^{[(K+1) \Delta]\wedge T} \abs { b(K\Delta,X_{K \Delta}^\epsilon,\hat{Y}_s^\epsilon)-b(s,X_{K \Delta}^\epsilon,\hat{Y}_s^\epsilon) }^2 ds \right ]  \\
 & + 3 \Delta \mathbb{E} \left [\sum_{K=0}^{\lceil T/ \Delta \rceil -1} \int_{K \Delta}^{[(K+1) \Delta]\wedge T} \abs { b(s,X_{K \Delta}^\epsilon,\hat{Y}_s^\epsilon)-b(s,X_{s}^{\epsilon},\hat{Y}_s^\epsilon ) }^2 ds \right ] \\
  & + 3 \Delta \mathbb{E} \left [\sum_{K=0}^{\lceil T/ \Delta \rceil -1} \int_{K \Delta}^{[(K+1) \Delta]\wedge T} \abs { b(s,X_{s}^{\epsilon},\hat{Y}_s^\epsilon )-b(s,X_{s}^{\epsilon},Y_s^{\epsilon} ) }^2 ds \right ] \\
 & \leq 3\Delta L^2 \sum_{K=0}^{\lceil T/ \Delta \rceil -1} \int_{K \Delta}^{[(K+1) \Delta]\wedge T} (s-K\Delta)^{2\gamma}+  \mathbb{E} \left [ \abs{ X_s^\epsilon - X_{K \Delta}^\epsilon }^2 \right ]  +\mathbb{E} \left  [ \abs{Y_s^\epsilon - \hat{Y}_s^{\epsilon} }^2\right]  ds
\end{align*}
By \eqref{eq_estimate_gronwall} and lemma \ref{lemma:lemma} it implies $E_1^\epsilon \to 0$ as $\epsilon \to 0$.\\
In a similar way we treat the term $E_2^\epsilon$. By Doob's inequality and Lipschitzianity we obtain: 
\begin{align*}
E_2^\epsilon 
& \leq 4 \mathbb{E} \left [ \sum_{K=0}^{\lceil T/ \Delta \rceil -1} \int_{K \Delta}^{[(K+1) \Delta]\wedge T} \abs { \sigma(K\Delta,X_{K \Delta}^\epsilon,\hat{Y}_s^\epsilon)-\sigma(s,X_{s}^{\epsilon},Y_s^{\epsilon} )  }^2 ds \right ]  \\
 & \leq 12L^2 \sum_{K=0}^{\lceil T/ \Delta \rceil -1} \int_{K \Delta}^{[(K+1) \Delta]\wedge T} (s-K\Delta)^{2\gamma}+  \mathbb{E} \left [ \abs{ X_s^\epsilon - X_{K \Delta}^\epsilon }^2 \right ]  +\mathbb{E} \left  [ \abs{Y_s^\epsilon - \hat{Y}_s^{\epsilon} }^2\right]  ds
\end{align*}
from which it follows $E_2^\epsilon \to 0$ as $\epsilon \to 0$. \\
This implies \eqref{eq2_th}.

Now we want to apply \cite[chapter 2: problem 4.11]{karatzas} to the family of the laws induced by the processes $\{(\hat{X}^\epsilon,W,\tilde{W})\}_{\epsilon}$ to show that it is tight.\\
Until now we have considered processes on a time horizon $T$ fixed. To apply the result we need to extend them from $[0,T]$ to $[0,+ \infty)$, by choosing 
\begin{equation}\label{eq:extension}
    Z_t(\omega)=Z_T(\omega)
\end{equation}
for $\omega \in \Omega, t \geq T$ where $Z$ is any of the processes we have considered.\\
In this way we keep the continuity of the trajectories and we can see $Z$ as a random variable with values on $(\mathcal{C}[0,+\infty)^d,\mathcal{B}(\mathcal{C}[0,+\infty)^d))$.\\
We remind that $\mathcal{C}[0,+\infty)^d$ is a complete separable metric space with the following metric:
\begin{align*}
\rho (y_1,y_2)=\sum_{n =0}^{+\infty} \frac{1}{2^n} \max_{0 \leq t \leq n} (\abs{y_1(t)-y_2(t)}\wedge 1) && y_1, y_2 \in \mathcal{C}[0,+\infty)^d
\end{align*}
We show that the hypotheses of \cite[chapter 2: problem 4.11]{karatzas} hold:\\
(i) $\sup_\epsilon \mathbb{E} \left[\abs{(\hat{X}_0^\epsilon,W_0,\tilde{W}_0)} ^2\right] = (|x_0|^2,0,0) $\\
(ii) Since $\mathbb{E}\left [\abs {W_t-W_s}^4  \right]\leq C \abs{t-s}^{2}$ for every $0 \leq s,t \leq T$ and by \eqref{eq:lemma_delta1} we have:
$$\sup_{\epsilon} \mathbb{E}\left [\abs {(\hat{X}_t^\epsilon,W_t,\tilde{W}_t)-(\hat{X}_s^\epsilon,W_s,\tilde{W}_s)}^{4} \right ] \leq C_T \abs{t-s}^{2}$$
 for every $0 \leq s,t \leq T$.\\
This implies the family of the laws induced by $\{(\hat{X}^\epsilon,W,\tilde{W})\}_{\epsilon}$ is tight in $\mathcal{C}[0,+\infty)^{2d+l}$. Then by Prohorov's theorem it is weakly compact. This means that any sequence  $\{(\hat{X}^{\epsilon_n},W,\tilde{W})\}_{\epsilon_n}$ has a subsequence $\{(\hat{X}^{\epsilon_{n'}},W,\tilde{W})\}_{\epsilon_{n'}}$ converging weakly in $\mathcal{C}[0,+\infty)^{2d+l}$. In virtue of \cite[theorem 2.6]{billingsley} if we prove that the limit is unique then the whole family $\{(\hat{X}^{\epsilon},W,\tilde{W})\}_{\epsilon}$ converges weakly in $\mathcal{C}[0,+\infty)^{2d+l}$.\\
Hence we take an arbitrary subsequence $\{(\hat{X}^{\epsilon_{n'}},W,\tilde{W})\}_{\epsilon_{n'}}$ converging weakly in $\mathcal{C}[0,+\infty)^{2d+l}$ to a certain process $(\hat{X}^0,W,\tilde{W})$.\\
Then thanks to \eqref{eq2_th} and extension \eqref{eq:extension} we have that $\{(X^{\epsilon_{n'}},\hat{X}^{\epsilon_{n'}},W,\tilde{W} )\}_{\epsilon_{n'}}$ converges weakly in $\mathcal{C}[0,+\infty)^{3d+l}$ to $(\hat{X}^{0},\hat{X}^0,W,\tilde{W})$ and by the Skorokhod's theorem we can find a new probability space $(\Omega^*, \mathcal{F}^*, \mathbb{P}^*)$  and random variables $\{(x^{\epsilon_{n'}},\hat{x}^{\epsilon_{n'}},w^{\epsilon_{n'}},\tilde{w}^{\epsilon_{n'}}) \}_{\epsilon_{n'}}$ defined on it with values in $\mathcal{C}[0,+\infty)^{3d+l}$ that have the same law of $\{(X^{\epsilon_{n'}},\hat{X}^{\epsilon_{n'}},W,\tilde{W} )\}_{\epsilon_{n'}}$ and such that $\{(x^{\epsilon_{n'}},\hat{x}^{\epsilon_{n'}},w^{\epsilon_{n'}},\allowbreak \tilde{w}^{\epsilon_{n'}}) \}_{\epsilon_{n'}}$ converges a.s. to a random variable $(\hat{x}^{0},\hat{x}^{0},w,\tilde{w})$ with values in $\mathcal{C}[0,+\infty)^{3d+l}$ that has the same law of $(\hat{X}^{0},\hat{X}^{0},W,\tilde{W})$.\\
For semplicity we rename $\{(x^{\epsilon},\hat{x}^{\epsilon},w^{\epsilon},\tilde{w}^{\epsilon}) \}_{\epsilon}$ the subsequence $\{(x^{\epsilon_{n'}},\hat{x}^{\epsilon_{n'}},w^{\epsilon_{n'}},\tilde{w}^{\epsilon_{n'}}) \}_{\epsilon_{n'}}$ converging a.s.

To obtain the uniqueness of the limit mentioned above and hence obtain the thesis of the theorem we prove that the law of the process $\hat{x}^{0}$ coincides with the law of the unique solution of \eqref{eq:main} or equivalently that it is the unique solution to the corresponding martingale problem. In order to do this we want to apply \cite[chapter 5: proposition 4.6]{karatzas}.\\
Hence we consider the probability $P$ induced by $\hat{x}^{0}$ on $(\mathcal{C}[0,+\infty)^d,\mathcal{B}(\mathcal{C}[0,+\infty)^d))$ and the filtration $ \mathfrak{F}_t=\mathcal{G}_{t^+}$ where $\{\mathcal{G}_{t} \}$ is the augmentation under $P$ of the canonical filtration $\mathcal{B}(\mathcal{C}[0,t)^d)$.\\
In this setting we consider the martingale problem associated with $\{\overline{\mathcal{A}}_t\}_{t \geq 0}$ defined by
\begin{align*}
\overline{\mathcal{A}}_t f (x) = \sum_{i=1}^d \overline{b}_{i} (t,x) \frac{\partial f(x)}{\partial x_{i}} + \frac1{2} \sum_{i, k=1}^d \overline{a}_{ik}(t,x) \frac{\partial^{2} f(x)}{\partial x_{i} \, \partial x_{j}} && f \in \mathcal{C}^2(\mathbb{R}^d)
\end{align*}
In particular the probability $P$ is a solution to the martingale problem associated with $\{\overline{\mathcal{A}}_t\}_{t \geq 0}$ if the process $\{M_t^f\}_{t \geq 0}$ defined by
\begin{align} \label{eq:martingale_pb}
M_t^f(\eta)=f(\eta(t))-f(\eta(0))-\int_{0}^{t} \overline{\mathcal{A}}_s f(\eta(s)) ds, && 0 \leq t < + \infty, \eta \in \mathcal{C}[0,+\infty)^d
\end{align}
is a continuous $\mathfrak{F}_t$-martingale under $P$ for every $f \in \mathcal{C}^2(\mathbb{R}^d)$.\\
By \cite[chapter 5: proposition 4.6]{karatzas} in order for $P$ to be a solution to the martingale problem it is enough to prove that the process $\{M_t^f\}_{t \geq 0}$ is a continuous $\mathfrak{F}_t$-martingale under $P$ for the choices $f(x)=x_i$ and $f(x)=x_i x_j$ for every $0 \leq i,j \leq d$.

Therefore we set $\mathfrak{F}'_t=\sigma \{w_s, \tilde{w}_s, s \leq t \}$, we fix the component $i\leq d$ and we show that the process $\{m_t^{x_i}\}_{t \geq 0}$ defined by
$$ m_t^{x_i}=\hat{x}_{t}^{0,i}-\hat{x}_{0}^{0,i}-\int_{0}^{t} \overline{b}_i(s,\hat{x}_{s}^0) ds $$ 
is a continuous martingale with respect to $\mathfrak{F}'_t$, $t \geq 0$. Notice that $\hat{x}_{0}^{0}=x_0$ a.s.\\
On $\left(\Omega^*,\mathcal{F}^*,\mathcal{F}_t^{'},w_t,\tilde{w}_t,\mathbb{P}^* \right)$ we consider system \eqref{eq:system} and we notice that the slow component of the solution is indistinguishable from $x^\epsilon_t$ defined above. Hence we denote by $(x_t^\epsilon,Y_t^\epsilon)$ the solution of system \eqref{eq:system} on $\left(\Omega^*,\mathcal{F}^*,\mathcal{F}_t^{'},w_t,\tilde{w}_t,\mathbb{P}^* \right)$.\\
Analogously $\hat{x}_t^\epsilon$ is indistinguishable from the slow component of the discretized process on\\ $\left(\Omega^*,\mathcal{F}^*,\mathcal{F}_t^{'},w_t,\tilde{w}_t,\mathbb{P}^* \right)$ which we then denote by $(\hat{x}_t^\epsilon,\hat{Y}_t^\epsilon)$ and we have:
\begin{equation}\label{eq_processf}
\hat {x}_{t}^{\epsilon,i}-\hat {x}_{K \Delta}^{\epsilon,i}-\int_{K \Delta}^{t} b_i(K\Delta,x_{K \Delta}^{\epsilon},\hat{Y}_s^{\epsilon})ds-\int_{K \Delta}^{t} [\sigma(K\Delta,x_{K \Delta}^{\epsilon},\hat{Y}_s^{\epsilon})dw_s]_i=0
\end{equation}
a.s. in $\Omega^*$, $\forall t \in [K \Delta, (K+1) \Delta)$. \\
Then the process
$$ \hat {x}_{t}^{\epsilon,i}-\hat {x}_{0}^{\epsilon,i}- \sum_{K=0}^{\lceil t/ \Delta \rceil -1} \int_{K \Delta}^{[(K+1) \Delta]\wedge t} b_i(K\Delta,x_{K \Delta}^{\epsilon},\hat{Y}_s^{\epsilon})ds $$ 
is a continuous local $\mathfrak{F}'_t$-martingale, $t \geq 0$. We need to pass this property to the limit as $\epsilon \to 0$.\\
Now we take $t_1 \leq t_2$ and we show that:
\begin{equation} \label{eq3_proof}
\lim_{\epsilon \to 0} \mathbb{E} \left [ \abs{\mathbb{E} \left [ \sum_{K=\lceil t_1 / \Delta \rceil }^{\lceil t_2/ \Delta \rceil -1} \int_{K \Delta }^{(K+1) \Delta} b_i(K\Delta,x_{K \Delta}^{\epsilon},\hat{Y}_s^{\epsilon})-\overline{b}_i(K\Delta,x_{K \Delta}^{\epsilon})ds \bigg|\mathfrak{F}'_{t_1}  \right ]} \right ]=0
\end{equation}
We define $\tilde{Y}_s^{\epsilon}$ as the time-changed process \begin{align}\label{eq:time_change}
    \tilde{Y}_s^{\epsilon}=\hat{Y}_{\epsilon s}^{\epsilon} && s \in [K \Delta / \epsilon, (K+1) \Delta / \epsilon)
\end{align} 
and we see that the processes $\tilde{Y}_{s+K \Delta / \epsilon}^{\epsilon}$ and $y_s^{(K\Delta,x_{K \Delta}^{\epsilon},Y_{K \Delta}^{\epsilon})}$ are indistinguishable on $[0, \Delta / \epsilon)$. Here $y_s^{(K\Delta,x_{K \Delta}^{\epsilon},Y_{K \Delta}^{\epsilon})}$ satisfies \eqref{eq:frozen} with $x_{K \Delta}^{\epsilon},Y_{K \Delta}^{\epsilon}$ frozen and Brownian motion \begin{align*}
    \Tilde{w}_{s}^K=\frac{1}{\sqrt{\epsilon}}(\Tilde{w}_{\epsilon s+K \Delta}-\Tilde{w}_{ K \Delta}) && s \geq 0
\end{align*}
By hypothesis \ref{hp:est}, \ref{hp:expectation}, by the Markov property and by \eqref{diseq1:lemma} we obtain \eqref{eq3_proof}, in fact:
\begin{small}
\begin{align*}
& \mathbb{E} \left [ \abs{\mathbb{E} \left [\sum_{K=\lceil t_1 / \Delta \rceil }^{\lceil t_2/ \Delta \rceil -1} \int_{K \Delta }^{(K+1) \Delta} b_i(K\Delta,x_{K \Delta}^{\epsilon},\hat{Y}_s^{\epsilon})-\overline{b}_i(K\Delta,x_{K \Delta}^{\epsilon})ds \bigg| \mathfrak{F}'_{t_1}  \right ]} \right ] \\ 
& = \mathbb{E} \left [ \abs{\mathbb{E} \left [ \sum_{K=\lceil t_1 / \Delta \rceil }^{\lceil t_2/ \Delta \rceil -1} \mathbb{E} \left [ \epsilon \int_{0}^{ \Delta / \epsilon  } b_i(K\Delta,x_{K \Delta}^{\epsilon},\tilde{Y}_{s+K \Delta / \epsilon}^{\epsilon})-\overline{b}_i(K\Delta,x_{K \Delta}^{\epsilon})ds \bigg| \mathfrak{F}'_{K \Delta}  \right ] \bigg| \mathfrak{F}'_{t_1}  \right ]} \right ] \\
& = \mathbb{E} \left [ \abs{\mathbb{E} \left [  \sum_{K=\lceil t_1 / \Delta \rceil }^{\lceil t_2/ \Delta \rceil -1} \mathbb{E} \left [ \epsilon   \int_{0}^{\Delta / \epsilon  } b_i(K\Delta,x_{K \Delta}^{\epsilon},y_s^{(K\Delta,x_{K \Delta}^{\epsilon},Y_{K \Delta}^{\epsilon})})-\overline{b}_i(K\Delta,x_{K \Delta}^{\epsilon})ds  \bigg| \mathfrak{F}'_{K \Delta}  \right ] \bigg| \mathfrak{F}'_{t_1}  \right ]} \right ] \\
& = \mathbb{E} \left [ \abs{\mathbb{E} \left [  \sum_{K=\lceil t_1 / \Delta \rceil }^{\lceil t_2/ \Delta \rceil -1} \mathbb{E} \left [ \epsilon   \int_{0}^{ \Delta / \epsilon  } b_i(K\Delta,x_{K \Delta}^{\epsilon},y_s^{(K\Delta,x_{K \Delta}^{\epsilon},Y_{K \Delta}^{\epsilon})})-\overline{b}_i(K\Delta,x_{K \Delta}^{\epsilon})ds  \bigg| \left(x_{K \Delta}^{\epsilon},Y_{K \Delta}^{\epsilon}\right)  \right ] \bigg| \mathfrak{F}'_{t_1}  \right ]} \right ] \\
& \leq \mathbb{E} \left [ \sum_{K=\lceil t_1 / \Delta \rceil }^{\lceil t_2/ \Delta \rceil -1} \Delta \overline{K} (\Delta / \epsilon) (1+|x_{K \Delta}^{\epsilon}|^2+|Y_{K \Delta}^{\epsilon}|^2)   \right ]\\
 &
\leq C \overline{K}(\Delta / \epsilon) \xrightarrow[\epsilon \to 0]{\Delta / \epsilon \to +\infty} 0
\end{align*}
\end{small}
By dominated convergence for uniform integrable functions for the second term on the right side it follows that
\begin{align*}
& \lim_{\epsilon \to 0} \mathbb{E} \left [ \abs{\mathbb{E} \left [ \sum_{K=\lceil t_1 / \Delta \rceil }^{\lceil t_2/ \Delta \rceil -1} \int_{K \Delta }^{(K+1) \Delta } b_i(K\Delta,x_{K \Delta}^{\epsilon},\hat{Y}_s^{\epsilon})-\overline{b}_i(s,\hat {x}_{s}^{0})ds \bigg| \mathfrak{F}'_{t_1}  \right ]} \right ] \\ 
& \leq \lim_{\epsilon \to 0}  \mathbb{E} \left [ \abs{\mathbb{E} \left [ \sum_{K=\lceil t_1 / \Delta \rceil }^{\lceil t_2/ \Delta \rceil -1} \int_{K \Delta }^{(K+1) \Delta } b_i(K\Delta,x_{K \Delta}^{\epsilon},\hat{Y}_s^{\epsilon})-\overline{b}_i(K\Delta,x_{K \Delta}^{\epsilon})ds \bigg| \mathfrak{F}'_{t_1}  \right ]} \right ] \\ 
& +\lim_{\epsilon \to 0}   \int_{\lceil t_1 / \Delta \rceil \Delta }^{\lceil t_2/ \Delta \rceil\Delta } \mathbb{E} \left [ \abs{ \overline{b}_i(\lfloor s / \Delta \rfloor \Delta,x_{\lfloor s / \Delta \rfloor \Delta}^{\epsilon})-\overline{b}_i(s,\hat {x}_{s}^{0})}\right ] ds  =0
\end{align*}
From this and by dominated convergence for uniform integrable functions we have:
\begin{align*}
& \lim_{\epsilon \to 0} \mathbb{E} \Bigg [ \hat {x}_{\lceil t_2 / \Delta \rceil \Delta}^{\epsilon,i}-\hat {x}_{\lceil t_1 / \Delta \rceil \Delta}^{\epsilon,i} - \sum_{K=\lceil t_1 / \Delta \rceil }^{\lceil t_2/ \Delta \rceil -1}\int_{K \Delta }^{(K+1) \Delta} b_i(K\Delta,x_{K \Delta}^{\epsilon},\hat{Y}_s^{\epsilon})ds \bigg| \mathfrak{F}'_{t_1}  \Bigg ]\\ 
& = \mathbb{E} \Bigg [ \hat {x}_{t_2}^{0,i}-\hat {x}_{t_1}^{0,i}-  \int_{t_1}^{t_2} \overline{b}_i(s,\hat {x}_{s}^{0})ds \bigg| \mathfrak{F}'_{t_1}  \Bigg ]
\end{align*}
where the convergence as $\epsilon \to 0$ is in $\mathcal{L}^1(\Omega^*)$.\\
By \eqref{eq_processf} the term
inside the expectation on the left-hand-side
  is a sum of stochastic integrals.\\ 
 It implies $\forall t_1 \leq t_2 $
\begin{align*}
 \mathbb{E} \left [ \hat {x}_{t_2}^{0,i}-\hat {x}_{t_1}^{0,i}-  \int_{t_1}^{t_2} \overline{b}_i(s,\hat {x}_{s}^{0})ds \bigg| \mathfrak{F}'_{t_1}  \right ]=0 && \mathbb{P}^*a.s.
\end{align*}
This implies that $\{m_t^{x_i} \}_{t \geq 0}$ is a continuous $\mathfrak{F}'_t$-martingale.\\
Then we have that $\mathfrak{F}'_t \supseteq \mathcal{F}_t^{\hat{x}^0}=\sigma \{\hat{x}_{s}^0,s \leq t \}$: in fact the process $\hat{x}^0_t$ is $\mathfrak{F}'_t $-measurable since it is the a.s. limit of the family $\{\hat{x}^{\epsilon} \}_{\epsilon}$ which is $\mathfrak{F}'_t $-measurable.
This implies that $\{m_t^{x_i} \}_{t \geq 0}$ is a continuous $\mathcal{F}_t^{\hat{x}^0}$-martingale. Finally it follows that the process $\{M_t^{x_i} \}_{t \geq 0}$ defined by \eqref{eq:martingale_pb} with $f(x)=x_i$ is a continuous $\mathfrak{F}_t$-martingale under $P$ for every $i \leq d$.

Now we show that the following process is a continuous martingale with respect to $\mathfrak{F}'_t$, $t\geq 0$
\begin{align*}
m_t^{x_i x_j}=\hat {x}_{t}^{0,i} \hat {x}_{t}^{0,j} - \hat {x}_{0}^{0,i} \hat {x}_{0}^{0,j} -\int_{0}^{t} \overline{a}_{ij}(s,\hat{x}_{s}^0) + \hat {x}_{s}^{0,i} \overline{b}_j(s,\hat{x}_{s}^0) + \hat {x}_{s}^{0,j} \overline{b}_i(s,\hat{x}_{s}^0) ds && \forall i,j \leq d
\end{align*}
For $t \in [K \Delta, (K+1) \Delta)$ we have:
\begin{align*}
\hat {x}_{t}^{\epsilon,i} \hat {x}_{t}^{\epsilon,j}-\hat {x}_{0}^{\epsilon,i} \hat {x}_{0}^{\epsilon,j} &=   \int_{0}^{t} a_{ij}(K\Delta,x_{K \Delta}^\epsilon,\hat{Y}_s^\epsilon) + \hat {x}_{s}^{\epsilon,i} b_j(K\Delta,x_{K \Delta}^\epsilon,\hat{Y}_s^\epsilon) + \hat {x}_{s}^{\epsilon,j} b_i(K\Delta,x_{K \Delta}^\epsilon,\hat{Y}_s^\epsilon) ds \\
& + \int_{0}^{t}  \hat {x}_{s}^{\epsilon,i} [\sigma(K\Delta,x_{K \Delta}^\epsilon,\hat{Y}_s^\epsilon)  dw_s]_j  + \int_{0}^{t}  \hat {x}_{s}^{\epsilon,j} [\sigma(K\Delta,x_{K \Delta}^\epsilon,\hat{Y}_s^\epsilon)  dw_s]_i
\end{align*}
As before we have that the following process is a continuous $\mathfrak{F}'_t$-martingale, $t\geq 0$

\begin{align*}
\hat {x}_{t}^{\epsilon,i} \hat {x}_{t}^{\epsilon,j}-\hat {x}_{0}^{\epsilon,i} \hat {x}_{0}^{\epsilon,j} -  \sum_{K=0}^{\lceil t/ \Delta \rceil -1} \int_{K \Delta}^{[(K+1) \Delta] \wedge t} & a_{ij}(K\Delta,x_{K \Delta}^\epsilon,\hat{Y}_s^\epsilon) + \hat {x}_{s}^{\epsilon,i} b_j(K\Delta,x_{K \Delta}^\epsilon,\hat{Y}_s^\epsilon)\\
& + \hat {x}_{s}^{\epsilon,j} b_i(K\Delta,x_{K \Delta}^\epsilon,\hat{Y}_s^\epsilon) ds
\end{align*}
and we need to pass this property to the limit process as $\epsilon \to 0$.\\
With the same argument of before we are able to show that $\forall i,j \leq d$, $\forall 0 \leq t_1 \leq t_2 $

\begin{equation} \label{eq_aux_3}
\lim_{\epsilon \to 0} \mathbb{E} \left [ \abs{\mathbb{E} \left [ \sum_{K=\lceil t_1 / \Delta \rceil }^{\lceil t_2/ \Delta \rceil -1} \int_{K \Delta }^{(K+1) \Delta} a_{ij}(K\Delta,x_{K \Delta}^{\epsilon},\hat{Y}_s^{\epsilon})-\overline{a}_{ij}(s,\hat {x}_{s}^{0})ds \bigg| \mathfrak{F}'_{t_1}  \right ]} \right ] =0
\end{equation}

\begin{equation} \label{eq_aux_4}
 \lim_{\epsilon \to 0} \mathbb{E} \left [ \abs{ \mathbb{E} \left [\hat {x}_{t}^{\epsilon,i} \hat {x}_{t}^{\epsilon,j} - \hat {x}_{t}^{0,i} \hat {x}_{t}^{0,j} \bigg| \mathfrak{F}'_{t_1}  \right ] } \right ] = 0 
\end{equation}
Now we treat the last term remaining and we show that $\forall i,j \leq d$, $ 0 \leq t_1 \leq t_2$ 
\begin{equation}\label{eq_aux_2}
 \lim_{\epsilon \to 0} \mathbb{E} \left [ \abs{\mathbb{E} \left [\sum_{K=\lceil t_1 / \Delta \rceil }^{\lceil t_2/ \Delta \rceil -1} \int_{K \Delta}^{(K+1) \Delta } \hat {x}_{s}^{\epsilon,j} b_i(K\Delta,x_{K \Delta}^{\epsilon},\hat{Y}_s^{\epsilon})-\hat {x}_{s}^{0,j} \overline{b}_i(s,\hat{x}_{s}^{0})ds \bigg| \mathfrak{F}'_{t_1}  \right ]} \right ]  =0
\end{equation}
First we prove:
\begin{small}
\begin{equation}\label{eq_aux}
\lim_{\epsilon \to 0}  \mathbb{E} \left [ \abs{\mathbb{E} \left [ \sum_{K=\lceil t_1 / \Delta \rceil }^{\lceil t_2/ \Delta \rceil -1} \int_{K \Delta }^{(K+1) \Delta }  \hat {x}_{s}^{\epsilon,j} ( b_i(K\Delta,x_{K \Delta}^{\epsilon},\hat{Y}_s^{\epsilon})-\overline{b}_i(K\Delta,x_{K \Delta}^{\epsilon}))ds \bigg| \mathfrak{F}'_{t_1}  \right ]} \right ] =0
\end{equation}
\end{small}
In fact we have:
\begin{small}
\begin{align*}
& \mathbb{E} \left [ \abs{\mathbb{E} \left [\sum_{K=\lceil t_1 / \Delta \rceil }^{\lceil t_2/ \Delta \rceil -1}  \int_{K \Delta }^{(K+1) \Delta } \hat {x}_{s}^{\epsilon,j} (b_i(K\Delta,x_{K \Delta}^{\epsilon},\hat{Y}_s^{\epsilon})-\overline{b}_i(K\Delta,x_{K \Delta}^{\epsilon})) ds \bigg| \mathfrak{F}'_{t_1}  \right ]} \right ] \\ 
& = \mathbb{E} \left [ \abs{\mathbb{E} \left [ \sum_{K=\lceil t_1 / \Delta \rceil }^{\lceil t_2/ \Delta \rceil -1} \mathbb{E} \left [ \epsilon \int_{K \Delta / \epsilon}^{(K+1) \Delta / \epsilon  } \hat {x}_{s \epsilon}^{\epsilon,j} (b_i(K\Delta,x_{K \Delta}^{\epsilon},\tilde{Y}_s^{\epsilon})-\overline{b}_i(K\Delta,x_{K \Delta}^{\epsilon}))ds \bigg| \mathfrak{F}'_{K \Delta}  \right ] \bigg| \mathfrak{F}'_{t_1}  \right ]} \right ] \\
& \leq \mathbb{E} \left [ \abs{\mathbb{E} \left [  \sum_{K=\lceil t_1 / \Delta \rceil }^{\lceil t_2/ \Delta \rceil -1} \mathbb{E} \left [ \epsilon   \int_{K \Delta / \epsilon}^{(K+1) \Delta / \epsilon  } (\hat {x}_{s \epsilon}^{\epsilon,j}-\hat{x}_{K \Delta}^{\epsilon,j})(b_i(K\Delta,x_{K \Delta}^{\epsilon},\tilde{Y}_s^{\epsilon})-\overline{b}_i(K\Delta,x_{K \Delta}^{\epsilon}))ds \bigg| \mathfrak{F}'_{K \Delta}  \right ] \bigg| \mathfrak{F}'_{t_1}  \right ]} \right ] \\
& + \mathbb{E} \left [ \abs{\mathbb{E} \left [  \sum_{K=\lceil t_1 / \Delta \rceil }^{\lceil t_2/ \Delta \rceil -1} \mathbb{E} \left [ \epsilon   \int_{K \Delta / \epsilon}^{(K+1) \Delta / \epsilon  } \hat{x}_{K \Delta}^{\epsilon,j} (b_i(K\Delta,x_{K \Delta}^{\epsilon},\tilde{Y}_s^{\epsilon})-\overline{b}_i(K\Delta,x_{K \Delta}^{\epsilon}))ds  \bigg| \mathfrak{F}'_{K \Delta} \right ] \bigg| \mathfrak{F}'_{t_1}  \right ]} \right ]
\end{align*}
\end{small}
\begin{flushleft}
We show that the first term on the right-hand-side tends to zero.\\
By Hölder inequality, sublinearity, \eqref{diseq1:lemma} and \eqref{eq:lemma_delta1} we have:
\end{flushleft}
\begin{small}
\begin{align*}
& \mathbb{E} \left [ \abs{\mathbb{E} \left [  \sum_{K=\lceil t_1 / \Delta \rceil }^{\lceil t_2/ \Delta \rceil -1} \mathbb{E} \left [ \epsilon   \int_{K \Delta / \epsilon}^{(K+1) \Delta / \epsilon  } (\hat {x}_{s \epsilon}^{\epsilon,j}-\hat {x}_{K \Delta}^{\epsilon,j})(b_i(K\Delta,x_{K \Delta}^{\epsilon},\tilde{Y}_s^{\epsilon})-\overline{b}_i(K\Delta,x_{K \Delta}^{\epsilon}))ds \bigg| \mathfrak{F}'_{K \Delta}  \right ] \bigg| \mathfrak{F}'_{t_1}  \right ]} \right ] \\
& \leq \sum_{K=\lceil t_1 / \Delta \rceil }^{\lceil t_2/ \Delta \rceil -1} \epsilon   \int_{K \Delta / \epsilon}^{(K+1) \Delta / \epsilon  } \mathbb{E} \left [ |\hat {x}_{s \epsilon}^{\epsilon,j}-\hat{x}_{K \Delta}^{\epsilon,j}|^2 \right]^{1/2}   \mathbb{E} \left[  | b_i(K\Delta,x_{K \Delta}^{\epsilon},\tilde{Y}_s^{\epsilon})-\overline{b}_i(K\Delta,x_{K \Delta}^{\epsilon})|^2 \right ]^{1/2} ds   \\
& \leq \sum_{K=\lceil t_1 / \Delta \rceil }^{\lceil t_2/ \Delta \rceil -1} \epsilon \int_{K \Delta / \epsilon}^{(K+1) \Delta / \epsilon  }   C_1 \Delta^{1/2} ds   \xrightarrow[]{\epsilon \to 0} 0 
\end{align*}
\end{small}
where $C_1$ is independent of $\epsilon$.\\
Similarly to \eqref{eq3_proof} also the second term on the right-hand-side goes to zero:
\begin{align*}
& \mathbb{E} \left [ \abs{\mathbb{E} \left [  \sum_{K=\lceil t_1 / \Delta \rceil }^{\lceil t_2/ \Delta \rceil -1} \mathbb{E} \left [ \epsilon   \int_{K \Delta / \epsilon}^{(K+1) \Delta / \epsilon  } \hat{x}_{K \Delta}^{\epsilon,j} (b_i(K\Delta,x_{K \Delta}^{\epsilon},\tilde{Y}_s^{\epsilon})-\overline{b}_i(K\Delta,x_{K \Delta}^{\epsilon}))ds  \bigg| \mathfrak{F}'_{K \Delta} \right ] \bigg| \mathfrak{F}'_{t_1}  \right ]} \right ]\\
&  \leq  \sum_{K=\lceil t_1 / \Delta \rceil }^{\lceil t_2/ \Delta \rceil -1} \Delta \overline{K} (\Delta / \epsilon) (\mathbb{E} \left [ |\hat{x}_{K \Delta}^{\epsilon}|\right]+\mathbb{E} \left [ |\hat{x}_{K \Delta}^{\epsilon}|^2\right]^{1/2} \mathbb{E} \left [ |x_{K \Delta}^{\epsilon}|^4\right]^{1/2}+\mathbb{E} \left [ |\hat{x}_{K \Delta}^{\epsilon}||Y_{K \Delta}^{\epsilon}|^2\right]) \\
&\leq C_2 \overline{K}(\Delta / \epsilon) \xrightarrow[\epsilon \to 0]{\Delta / \epsilon \to +\infty} 0
\end{align*}
where $C_2$ is independent of $\epsilon$ and by hypothesis \ref{hp:expectation} we have
\begin{align*}
\mathbb{E} \left [ |\hat{x}_{K \Delta}^{\epsilon}||Y_{K \Delta}^{\epsilon}|^2\right] 
= \mathbb{E} \left [ |Y_{K \Delta}^{\epsilon}|^2 \mathbb{E} \left [  |\hat{x}_{K \Delta}^{\epsilon}| \big | Y_{K \Delta}^{\epsilon} \right]\right] \leq C_3 \mathbb{E} \left [ |Y_{K \Delta}^{\epsilon}|^2 \right] \leq C_4
\end{align*}
with $C_3,C_4$ independent of $\epsilon$. Notice that $\mathbb{E} \left [  |\hat{x}_{K \Delta}^{\epsilon}| \big | Y_{K \Delta}^{\epsilon} \right]\leq C_3$, in fact:
\begin{align*}
 \mathbb{E} \left [  |\hat{x}_{K \Delta}^{\epsilon}| \big | Y_{K \Delta}^{\epsilon} \right] &= \mathbb{E} \left [\mathbb{E} \left [  |\hat{x}_{K \Delta}^{\epsilon}| \big | Y_{s}^{\epsilon}, s \leq K \Delta \right]\big | Y_{K \Delta}^{\epsilon} \right] \\
 & \leq  \mathbb{E} \left [\mathbb{E} \left [  |\hat{x}_{K \Delta}^{\epsilon}|^2 \big | Y_{s}^{\epsilon}, s \leq K \Delta \right]\big | Y_{K \Delta}^{\epsilon} \right] ^{1/2}
\end{align*}
and $\mathbb{E} \left [  |\hat{x}_{K \Delta}^{\epsilon}|^2 \big | Y_{s}^{\epsilon}=y_s, s \leq K \Delta \right]\leq C_3$ because we can repeat the same proof of \eqref{diseq2:lemma} with any deterministic continuous function $y_s$. 
This implies \eqref{eq_aux}.\\
As before by \eqref{eq_aux} and dominated convergence we have \eqref{eq_aux_2}.\\
Conditions \eqref{eq_aux_3}, \eqref{eq_aux_4}, \eqref{eq_aux_2} imply that the process $\{m_t^{x_i x_j} \}_{t \geq 0}$ is a continuous martingale with respect to $\mathfrak{F}'_t$, $t \geq 0$ and then the process $\{M_t^{x_i x_j} \}_{t \geq 0}$ defined by \eqref{eq:martingale_pb} with $f(x)=x_i x_j$ is a continuous $\mathfrak{F}_t$-martingale under $P$ for every $i,j \leq d$.

At this point we can apply \cite[chapter 5: proposition 4.6]{karatzas} and say that there exists a d-dimensional $\hat{\mathfrak{F}}_t$-Brownian motion $\{ w^0_t \}_{t \geq 0}$ defined on an extension $(\hat{\Omega},\hat{\mathfrak{F}},\hat{\mathbb{P}})$ of $(\mathcal{C}[0,+\infty)^d,\allowbreak \mathcal{B}(\mathcal{C}[0,+\infty)^d), P)$ such that for $\eta \in \mathcal{C}[0,+\infty)^d$, we have that $(\hat{\Omega},\hat{\mathfrak{F}},\hat{\mathfrak{F}}_t,\hat{\mathbb{P}},X_t \coloneqq \eta(t),w^0_t) $ is a weak solution to \eqref{eq:main}. Moreover thanks to remark \ref{rem:uniqueness_averaged_eq} the solution of \eqref{eq:main} is unique. Then by \cite[ chapter 5]{karatzas} the probability P is  the unique solution to the martingale problem associated with $\{\overline{\mathcal{A}}_t\}_{t \geq 0}$ and fixed initial condition $P(\eta \in \mathcal{C}[0,+\infty)^d,\eta(0)=x_0)=1$. By \cite[theorem 2.6]{billingsley} it follows that the whole family $\{ \hat{X}^\epsilon \}_\epsilon$ converges weakly in $\mathcal{C}[0,+\infty)^d$ to the unique strong solution of \eqref{eq:main} and this implies the weak convergence in $\mathcal{C}[0,T]^d$. By \eqref{eq2_th} we have the thesis of the theorem.
\end{proof}

\section{Explicit conditions}
\label{sec:Explicit conditions}
In this section following the idea of remark \ref{rem:candidates} we provide some explicit conditions under which the implicit hypotheses \ref{hp:est}, \ref{hp:expectation} of the previous section are verified. In this way we can prove the validity of the averaging principle under only explicit conditions, in a similar way to what happens in \cite{Wei_liu} but with $\sigma$ dependent on $Y^\epsilon$. 

Let's consider system \eqref{eq:system} under hypotheses \ref{hp:lip}, \ref{hp:sub}. Moreover we assume the following non-degeneracy and dissipativity conditions:
\begin{hypothesis}[Non-degeneracy] \label{hp:pos} 
\begin{equation}
\inf_{0 \leq t \leq T} \inf_{x,y} \sigma (t,x,y)>0
\end{equation}
\end{hypothesis}
\begin{hypothesis}[Dissipativity] \label{hp:dis} 
There exists a constant $\beta >0 $ such that
\begin{equation*}
<B(t,x,y_2)-B(t,x,y_1),y_2-y_1>+|C(t,x,y_2)-C(t,x,y_1)|^2 \leq -\beta |y_2-y_1|^2
\end{equation*}
for every $ 0 \leq t \leq T$, $x \in \mathbb{R}^d$, $y_1,y_2 \in \mathbb{R}^l$.
\end{hypothesis}
Under these conditions we show that hypotheses \ref{hp:est}, \ref{hp:expectation} are satisfied. We start by proving the validity of the following standard lemmas:
\begin{lemma}  \label{lemma:aux_est2}
\begin{equation}
\mathbb{E}\left[\abs{y_s^{t,x,y_1}-y_s^{t,x,y_2}}^2\right] \leq e^{-2\beta s} |y_1-y_2|^2
\end{equation}
for every $0 \leq s$, $0\leq t \leq T$, $x \in \mathbb{R}^d$, $y_1,y_2 \in \mathbb{R}^l$.
\end{lemma}
\begin{proof}
By Ito's formula and hypothesis \ref{hp:dis} we have:
\begin{align*}
\frac{d}{ds} \mathbb{E}\left[\abs{y_s^{t,x,y_1}-y_s^{t,x,y_2}}^2\right] & =  \mathbb{E} \left [2 <B(t,x,y_s^{t,x,y_1})-B(t,x,y_s^{t,x,y_2}),y_s^{t,x,y_1}-y_s^{t,x,y_2}> \right ] \\ 
& +\mathbb{E} \left [|C(t,x,y_s^{t,x,y_1})-C(t,x,y_s^{t,x,y_2})|^2 \right ] \\
& \leq  -2\beta  \mathbb{E} \left [\abs{y_s^{t,x,y_1}-y_s^{t,x,y_2}}^2 \right] 
\end{align*}
Then by the comparison theorem we have the thesis.
\end{proof}

\begin{lemma} \label{lemma:aux_est}
There exists a constant $C>0$ such that
\begin{equation*}
\mathbb{E}\left[ \abs{y_s^{t_1,x_1,y}-y_s^{t_2,x_2,y}}^2 \right] \leq C \left(|t_1-t_2|^{2\gamma}+|x_1-x_2|^2 \right)
\end{equation*}
for every $0  \leq s$, $0\leq t_1,t_2 \leq T$, $x_1,x_2 \in \mathbb{R}^d$, $y \in \mathbb{R}^l$.
\end{lemma}
\begin{proof}
By hypothesis \ref{hp:dis} we have:

\begin{align*}
\mathbb{E}\left[ \abs{y_s^{t_1,x_1,y}-y_s^{t_2,x_2,y}}^2 \right]  & = \int_0^s \mathbb{E} [2 <B(t_1,x_1,y_r^{t_1,x_1,y})-B(t_2,x_2,y_r^{t_2,x_2,y}),y_r^{t_1,x_1,y}-y_r^{t_2,x_2,y}>]dr \\ 
& +\int_0^s   \mathbb{E} [|C(t_1,x_1,y_r^{t,x_1,y})-C(t_2,x_2,y_r^{t,x_2,y})|^2]dr \\
& \leq \int_0^s 2\mathbb{E} [ <B(t_1,x_1,y_r^{t_1,x_1,y})-B(t_1,x_1,y_r^{t_2,x_2,y}),y_r^{t_1,x_1,y}-y_r^{t_2,x_2,y}>]dr \\ 
& +\int_0^s 2 \mathbb{E} [|C(t_1,x_1,y_r^{t_1,x_1,y})-C(t_1,x_1,y_r^{t_2,x_2,y})|^2]dr \\
& + \int_0^s 2\mathbb{E} [ <B(t_1,x_1,y_r^{t_2,x_2,y})-B(t_2,x_2,y_r^{t_2,x_2,y}),y_r^{t_1,x_1,y}-y_r^{t_2,x_2,y}>]dr \\ 
& +\int_0^s 2  \mathbb{E} [|C(t_1,x_1,y_r^{t_2,x_2,y})-C(t_2,x_2,y_r^{t_2,x_2,y})|^2]dr \\
& \leq   \int_0^s -2\beta \mathbb{E}\left[ \abs{y_r^{t_1,x_1,y}-y_r^{t_2,x_2,y}}^2 \right]  +2 L  \mathbb{E}\left[ \abs{y_r^{t_1,x_1,y}-y_r^{t_2,x_2,y}} \right] \\ &\times \left(|t_1-t_2|^\gamma+|x_1-x_2|\right) 
 + C_1 \left(|t_1-t_2|^{2\gamma}+|x_1-x_2|^2 \right)  dr
\end{align*}
for some $C_1>0$.\\
Then by Young's inequality we have:
\begin{align*}
\mathbb{E}\left[ \abs{y_s^{t_1,x_1,y}-y_s^{t_2,x_2,y}}^2 \right]  & \leq    \int_0^s -\beta \mathbb{E}\left[ \abs{y_r^{t_1,x_1,y}-y_r^{t_2,x_2,y}}^2 \right] +  C_2 \left(|t_1-t_2|^{2\gamma}+|x_1-x_2|^2 \right) dr  
\end{align*}
for some $C_2>0$.\\
By the comparison theorem this implies the thesis.
\end{proof}
\begin{lemma}\label{lemma:moment_frozen_y}
There exists a constant $C>0$ such that
\begin{equation}
  \mathbb{E}\left[\abs{y_s^{t,x,y}}^2 \right] \leq e^{- \beta s} |y|^2  + C (1+|x|^2)
\end{equation}
for every $0  \leq s$, $0\leq t \leq T$, $x \in \mathbb{R}^d$, $y \in \mathbb{R}^l$.
\end{lemma}
\begin{proof}
By Ito's formula and hypothesis \ref{hp:dis} we obtain:
\begin{align*}
\frac{d}{ds}   \mathbb{E} \left[\abs{y_s^{t,x,y}}^2 \right]  & =  \mathbb{E} \left[ 2 <y_s^{t,x,y}, B(t,x, y_s^{t,x,y}) > + |C(t,x,y_s^{t,x,y})|^2\right]  \\
& \leq - 2 \beta \mathbb{E} \left[ |y_s^{t,x,y}|^2 \right]   + 2 \mathbb{E} \left[<y_s^{t,x,y},B(t,x,0)>+|C(t,x,0)|^2 \right]  
\end{align*}
By Young's inequality and the sublinearity of $B$ and $C$ we have:
\begin{align*}
\frac{d}{ds}  \mathbb{E} \left[\abs{y_s^{t,x,y}}^2 \right]
& \leq - \beta \mathbb{E} \left[ \abs{y_s^{t,x,y}}^2 \right]   + C_1   (1+|x|^2) 
\end{align*}
for some $C_1>0$.\\
Then by the comparison theorem we have the thesis.
\end{proof}
Now we are ready to prove that hypotheses \ref{hp:est}, \ref{hp:expectation} are satisfied:
\begin{lemma}\label{lemma:ergodicity}
Hypothesis \ref{hp:est} holds.
\end{lemma}
\begin{proof}
By lemma \ref{lemma:moment_frozen_y} for fixed $0 \leq t \leq T, x \in \mathbb{R}^d$ the family of the laws induced by $\{y_s^{t,x,y} \}_{s \geq 0}$ is tight in $\mathbb{R}^l$. Then by the Krylov-Bogoliubov theorem 
the semigroup
\begin{align*}
\mathcal{P}^{t,x}_s\phi(y)=\mathbb{E}[\phi(y_s^{t,x,y})] && \phi \in B_B(\mathbb{R}^l),y \in \mathbb{R}^l, s \geq 0
\end{align*}
has an invariant measure $\mu^{t,x}$.\\
Moreover there exists a constant $C>0$ such that
\begin{equation}\label{eq:moment_mu}
    \int_{\mathbb{R}^l} |z|^2 \mu^{t,x}(dz) \leq C (1+|x|^2)
\end{equation}
for every $0 \leq t \leq T$, $x \in \mathbb{R}^d$.\\
In fact since $\mu^{t,x}$ is invariant measure and by lemma \ref{lemma:moment_frozen_y} we have for every $s >0$
\begin{align*}
    \int_{\mathbb{R}^l} |z|^2 \mu^{t,x}(dz) & = \int_{\mathbb{R}^l} \mathbb{E}\left[  \abs{y_s^{t,x,z}}^2\right] \mu^{t,x}(dz)\\
    & \leq e^{-\beta s}  \int_{\mathbb{R}^l} |z|^2 \mu^{t,x}(dz) + \overline{C}(1+|x|^2)
\end{align*}
By taking $s$ large enough we have \eqref{eq:moment_mu}.\\
Then by lemma \ref{lemma:aux_est2} and \eqref{eq:moment_mu} we have the exponential decay to equilibrium:\\
let $\phi \in B_B(\mathbb{R}^l)$ be Lipschitz with Lipschitz constant $L(\phi)$
\begin{align*}
\abs{\mathcal{P}^{t,x}_s\phi(y) - \int_{\mathbb{R}^l} \phi(z)\mu^{t,x}(dz)} & = \abs{\int_{\mathbb{R}^l} \mathbb{E}\left[ \phi(y_s^{t,x,y}) -  \phi(y_s^{t,x,z})\right ] \mu^{t,x}(dz)} \\
& \leq L(\phi) \int_{ \mathbb{R}^l} \mathbb{E}\left[ \abs{ y_s^{t,x,y}- y_s^{t,x,z} }\right ] \mu^{t,x}(dz) \\
& \leq L(\phi) e^{-2\beta s} \int_{ \mathbb{R}^l} \abs{ y-z }\mu^{t,x}(dz) \\
& \leq L(\phi) e^{-2\beta s} \left(|y|+\int_{ \mathbb{R}^l} \abs{z }\mu^{t,x}(dz) \right)\\
& \leq C_1 L(\phi) e^{-2\beta s} (1+|x|+|y|)
\end{align*}
for some $C_1>0$.\\
This implies also the uniqueness of the invariant measure $\mu^{t,x}$.\\
Then by choosing $\phi(y)=b(t,x,y_s^{t,x,y})$ for fixed $t \geq 0, x \in \mathbb{R}^d$ we have 
\begin{align}\label{eq:est_b_b_bar}
\abs{\mathbb{E}\left[ b(t,x,y_s^{t,x,y})\right ] - \int_{\mathbb{R}^l} b(t,x,z)\mu^{t,x}(dz)} \leq C_1 L e^{-2\beta s} (1+|x|+|y|)
\end{align}
Integrating with respect to $[t_0,t_0+\tau]$ for $t_0,\tau >0$ and dividing by $\tau$ it implies \eqref{eq:erg1} with 
\begin{equation}
\label{eq:definition_bbar}
\overline{b}(t,x)=\int_{\mathbb{R}^l} b(t,x,z)\mu^{t,x}(dz)
\end{equation}
Now we notice that $a(t,x,y)=\sigma(t,x,y)\sigma(t,x,y)^T/2$ satisfies the condition:
\begin{equation}\label{eq:regularity_a}
|a(t_2,x_2,y_2)-a(t_1,x_1,y_1)|\leq \Tilde{C} (|t_2-t_1|^\gamma+|x_2-x_1|+|y_2-y_1| )(1+|x_1|+|x_2|)
\end{equation}
for some $\Tilde{C}>0$, for every $0 \leq t_1,t_2 \leq T$, $x_1,x_2 \in \mathbb{R}^d$, $y_1,y_2 \in \mathbb{R}^l$.\\
Then similarly to the previous calculation we have:
\begin{align*}
\abs{\mathbb{E}\left[ a(t,x,y_s^{t,x,y})\right ] - \int_{\mathbb{R}^l} a(t,x,z)\mu^{t,x}(dz)} & = \abs{\int_{\mathbb{R}^l} \mathbb{E}\left[ a(t,x,y_s^{t,x,y}) -  a(t,x,y_s^{t,x,z})\right ] \mu^{t,x}(dz)} \\
& \leq \Tilde{C} (1+|x|) \int_{ \mathbb{R}^l} \mathbb{E}\left[ \abs{ y_s^{t,x,y}- y_s^{t,x,z} }\right ] \mu^{t,x}(dz) \\
& \leq \Tilde{C} e^{-2\beta s} (1+|x|) \int_{ \mathbb{R}^l} \abs{ y-z }\mu^{t,x}(dz) \\
& \leq \Tilde{C} e^{-2\beta s} (1+|x|) \left(|y|+\int_{ \mathbb{R}^l} \abs{z }\mu^{t,x}(dz) \right)\\
& \leq C_2 e^{-2\beta s} (1+|x|^2+|y|^2)
\end{align*}
for some $C_2>0$.\\
Now integrating with respect to $[t_0,t_0+\tau]$ for $t_0,\tau >0$ and dividing by $\tau$ it implies \eqref{eq:erg2} with 
\begin{equation*}
    \overline{a}(t,x)=\int_{\mathbb{R}^l} a(t,x,z)\mu^{t,x}(dz)
\end{equation*}
and 
\begin{equation}
\label{eq:definition_sigma_bar}
    \overline{\sigma}(t,x)=\sqrt{2\overline{a}(t,x)}
\end{equation}
chosen as the unique symmetric positive definite square root.\\
Finally we show that:
\begin{equation}\label{eq:lip_b_bar}
    \abs{\overline{b}(t_2,x_2)-\overline{b}(t_1,x_1)} \leq C \left(|t_2-t_1|^\gamma + |x_2-x_1|  \right) 
\end{equation}
\begin{equation}\label{eq:lip_sigma_bar}
    \abs{\overline{\sigma}(t_2,x_2)-\overline{\sigma}(t_1,x_1)} \leq C \left(|t_2-t_1|^\gamma + |x_2-x_1|  \right) \left(1+|x_1|+|x_2|\right)
\end{equation}
for some $C>0$. This concludes the proof.\\
By \eqref{eq:est_b_b_bar} and by lemma \ref{lemma:aux_est} we have:
\begin{small}
\begin{align*}
\abs{\overline{b}(t_2,x_2)-\overline{b}(t_1,x_1)} & \leq  \abs{\int_{\mathbb{R}^l} b(t_2,x_2,z)\mu^{t_2,x_2}(dz)- \mathbb{E}\left [b(t_2,x_2,y_s^{t_2,x_2,0}) \right]} \\
& +\abs{\int_{\mathbb{R}^l} b(t_1,x_1,z)\mu^{t_1,x_1}(dz)- \mathbb{E}\left [b(t_1,x_1,y_s^{t_1,x_1,0}) \right]} \\
& + \abs{\mathbb{E}\left [b(t_2,x_2,y_s^{t_2,x_2,0}) - b(t_1,x_1,y_s^{t_1,x_1,0}) \right]} \\
& \leq C_3 L e^{- 2\beta s} (1+|x_1|+|x_2|) + L \left(|t_2-t_1|^\gamma + |x_2-x_1|  \right) + L \mathbb{E}\left [\abs{ y_s^{t_2,x_2,0} - y_s^{t_1,x_1,0} } \right] \\
& \leq C_3 L e^{- 2\beta s} (1+|x_1|+|x_2|) + C_4 \left(|t_2-t_1|^\gamma + |x_2-x_1|  \right) 
\end{align*}
\end{small}
for some $C_3,C_4>0$.\\
Letting $s \to +\infty$ we have \eqref{eq:lip_b_bar}.\\
With a similar calculation by \eqref{eq:regularity_a} we can show that $\overline{a}$ satisfies:
\begin{equation*}
    \abs{\overline{a}(t_2,x_2)-\overline{a}(t_1,x_1)} \leq C_5 \left(|t_2-t_1|^\gamma + |x_2-x_1|  \right) \left(1+|x_1|+|x_2|\right)
\end{equation*}
for some $C_5>0$.\\
Since $\overline{\sigma}$ is symmetric we consider the equality
\begin{align*}
(\overline{\sigma}(t_2,x_2)-\overline{\sigma}(t_1,x_1)) &  = (\overline{\sigma}(t_2,x_2)^2-\overline{\sigma}(t_1,x_1)^2) (\overline{\sigma}(t_2,x_2)+\overline{\sigma}(t_1,x_1))^{-1}\\
& = 2(\overline{a}(t_2,x_2)-\overline{a}(t_1,x_1)) (\overline{\sigma}(t_2,x_2)+\overline{\sigma}(t_1,x_1))^{-1}
\end{align*}
As $\overline{\sigma}$ is invertible with bounded inverse by hypothesis \ref{hp:pos} we have \eqref{eq:lip_sigma_bar}.
\end{proof}
\begin{lemma}\label{lemma:hp_bound}
Hypothesis \ref{hp:expectation} holds.
\end{lemma}
\begin{proof}
By Ito's formula and hypothesis $\ref{hp:dis}$ we have:
\begin{align*}
\frac{d}{ds}   \mathbb{E} \left[|Y_s^\epsilon|^2 \right]  & =  \mathbb{E} \left[ \frac{2}{\epsilon} <Y_s^\epsilon, B(s,X_s^\epsilon, Y_s^\epsilon) > + \frac{1}{\epsilon}  |C(s,X_s^\epsilon,Y_s^\epsilon)|^2\right]  \\
& \leq - \frac{2\beta}{\epsilon} \mathbb{E} \left[ |Y_s^\epsilon|^2 \right]   + \frac{2}{\epsilon} \mathbb{E} \left[<Y_s^\epsilon,B(s,X_s^\epsilon,0)>+|C(s,X_s^\epsilon,0|^2 \right]  
\end{align*}
Then by Young's inequality, the sublinearity of $B$ and $C$ and \eqref{diseq1:lemma} we obtain:
\begin{align*}
\frac{d}{ds}  \mathbb{E} \left[|Y_s^\epsilon|^2 \right]
& \leq - \frac{\beta}{\epsilon}  \mathbb{E} \left[ |Y_s^\epsilon|^2 \right]   + \frac{C_1}{\epsilon}  \mathbb{E} \left[ (1+|X_s^\epsilon|^2) \right] \\
& \leq - \frac{\beta}{\epsilon}  \mathbb{E} \left[ |Y_s^\epsilon|^2 \right]   + \frac{C_2}{\epsilon} 
\end{align*}
for suitable $C_1, C_2>0$ independent of $\epsilon$.\\
Then by the comparison theorem we have the thesis.
\end{proof}
Thus we have shown validity of the averaging principle under only explicit conditions:
\begin{theorem}[Averaging Principle II] \label{th:averaging2}
Assume that hypotheses \ref{hp:lip}, \ref{hp:sub}, \ref{hp:pos}, \ref{hp:dis} hold. Then  $ X^{\epsilon} \to \overline{X}$ weakly in  $\mathcal{C}[0,T]^d$ as $\epsilon \to 0$ where the process $\overline{X}$ is the unique strong solution of the averaged equation

\begin{align*} 
d\overline{X} _t = \overline{b}(t,\overline{X}_t)dt+\overline{\sigma}(t,\overline{X}_t)dW_t && \overline{X}_0=x_0, \forall t \in[0,T]
\end{align*}
and the coefficients $\overline{b}$, $\overline{\sigma}$ are defined by \eqref{eq:definition_bbar}, \eqref{eq:definition_sigma_bar} respectively.
\end {theorem}

Since it will be useful in the financial application we prove that the averaging principle holds also when the fast equation is perturbed by a slower term of the form $\epsilon^{-\eta}D(t,X_t^{\epsilon},Y_t^{\epsilon})$ for $0\leq\eta<1$. Indeed let's consider the following perturbed system for $0 \leq t \leq T$
\begin{equation} \label{eq:perturbed_d}
\begin{cases}
dX_t^{\epsilon}=b(t,X_t^{\epsilon},Y_t^{\epsilon})dt+\sigma(t,X_t^{\epsilon},Y_t^{\epsilon})dW_t \\
X_0^{\epsilon}=x_0 \\
dY_t^{\epsilon}=[\epsilon^{-1} B(t,X_t^{\epsilon},Y_t^{\epsilon})+\epsilon^{-\eta}D(t,X_t^{\epsilon},Y_t^{\epsilon})]dt+\epsilon^{-1/2}C(t,X_t^{\epsilon},Y_t^{\epsilon})d\tilde{W}_t \\
Y_0^{\epsilon}=y_0
\end{cases}
\end{equation}
In this case the discretized process for $t \in [K \Delta, ((K+1) \Delta) \wedge T)$ becomes:
\begin{equation*}
\begin{cases}
d\hat {X}_t^{\epsilon}=b(K \Delta, X_{K \Delta}^{\epsilon},\hat{Y}_t^{\epsilon})dt+\sigma(K \Delta,X_{K \Delta}^{\epsilon},\hat{Y}_t^{\epsilon})dW_t  \\
\hat {X}_{0}^{\epsilon}=x_0 \\
d\hat{Y}_t^{\epsilon}=\left[\epsilon^{-1} B(K \Delta,X_{K \Delta}^{\epsilon},\hat{Y}_t^{\epsilon})+\epsilon^{-\eta}D(K \Delta, X_{K \Delta}^{\epsilon},\hat{Y}_t^{\epsilon})\right]dt+\epsilon^{-1/2}C(K \Delta,X_{K \Delta}^{\epsilon},\hat{Y}_t^{\epsilon})d\tilde{W}_t \\
\hat {Y}_{K \Delta}^{\epsilon}=Y_{K \Delta}^{\epsilon} 
\end{cases}
\end{equation*}
Analogously we have the following proposition:
\begin{proposition}\label{prop:extension2}
 The conclusions of theorem \ref{th:averaging2} hold also when we consider system \eqref{eq:perturbed_d} under hypotheses \ref{hp:lip}, \ref{hp:sub}, \ref{hp:pos}, \ref{hp:dis} and
\begin{align*} 
\abs{D(t,x_2,y_2)-D(s,x_1,y_1)} & \leq L (|t-s|^\gamma+|x_2-x_1|+|y_2-y_1|)
\end{align*}
\begin{align*} 
\abs{D(t,x,y)} & \leq M (1+\abs{x}+\abs{y})
\end{align*}
for every $0 \leq t,s \leq T$, $x,x_1,x_2 \in \mathbb{R}^d$, $y,y_1,y_2 \in \mathbb{R}^l$.
\end{proposition}
Let's observe that in this context it doesn't make sense anymore to talk about "frozen equation" for \eqref{eq:frozen}. Anyway, since we are going to prove \eqref{eq:assumption_tilde_frozen}, we continue to refer to it in this way.

\begin{proof}
\textit{Step 1}: we prove that the hypotheses of theorem \ref{th:averaging} are verified. In fact lemmas \ref{lemma:aux_est2}, \ref{lemma:aux_est}, \ref{lemma:moment_frozen_y} and \ref{lemma:ergodicity} hold with the same proof since the frozen equation is the same as before. For what concerns lemma \ref{lemma:hp_bound} consider the inequality:
\begin{align*}
\frac{d}{ds}   \mathbb{E} \left[|Y_s^\epsilon|^2 \right] 
& \leq - \frac{2\beta}{\epsilon} \mathbb{E} \left[ |Y_s^\epsilon|^2 \right]   + \frac{2}{\epsilon} \mathbb{E} \left[<Y_s^\epsilon,B(s,X_s^\epsilon,0)>+|C(s,X_s^\epsilon,0|^2 \right]  \\
& +\frac{2}{\epsilon^\eta}\mathbb{E} \left [ <Y_s^\epsilon, D(s,X_s^\epsilon, Y_s^\epsilon) > \right] 
\end{align*}
Since $\eta <1$ the dominating terms are the ones with coefficient of order $\frac{1}{\epsilon}$. This implies the thesis of the lemma by a similar argument as before.

\textit{Step 2}: we prove the following estimate:
\begin{equation}\label{eq:assumption_tilde_frozen}
    \sup_{0 \leq s < \Delta/\epsilon}  \mathbb{E} \left [  \abs{\tilde{Y}_{s + K \Delta/\epsilon}^{\epsilon}-y_s^{(K\Delta,x_{K \Delta}^{\epsilon},Y_{K \Delta}^{\epsilon})}}^2 \right ] \leq C \epsilon^{\nu}
\end{equation}
for some $\nu, C>0$, for every $\epsilon>0$, $K=0...\lfloor T \rfloor$.\\
Here $y_s^{(K\Delta,x_{K \Delta}^{\epsilon},Y_{K \Delta}^{\epsilon})}$ satisfies \eqref{eq:frozen} with $x_{K \Delta}^{\epsilon},Y_{K \Delta}^{\epsilon}$ frozen  and Brownian motion 
\begin{align*}
    \Tilde{W}_{s}^K=\frac{1}{\sqrt{\epsilon}}(\Tilde{W}_{\epsilon s + K \Delta}-\Tilde{W}_{ K \Delta  }) && s \geq 0
\end{align*}
and $\tilde{Y}_s^{\epsilon}$ as usual is defined by
\begin{align}
    \tilde{Y}_s^{\epsilon}=\hat{Y}_{\epsilon s}^{\epsilon} && s \in [K \Delta / \epsilon, (K+1) \Delta / \epsilon)
\end{align} 
First we notice that now $\tilde{Y}_s^{\epsilon}$ satisfies:
\begin{align*}
    d\tilde{Y}_{s + K \Delta/\epsilon}^{\epsilon}=\left[ B(K \Delta,X_{K \Delta}^{\epsilon},\tilde{Y}_s^{\epsilon})+\epsilon^{1-\eta}D\left(K \Delta, X_{K \Delta}^{\epsilon},\tilde{Y}_s^{\epsilon})\right]dt+C(K \Delta,X_{K \Delta}^{\epsilon},\tilde{Y}_s^{\epsilon}\right)d\tilde{W}^K_s && \tilde{Y}_{K \Delta/\epsilon}^{\epsilon}=Y_{K \Delta}^\epsilon
\end{align*} 
for every $s \in [0, \Delta / \epsilon)$, for every $K$.\\
Then for $s\in [0,\Delta /\epsilon)$ and $K$, fixed by the dissipativity we have:
\begin{small}
\begin{align*}
& \frac{d}{ds} \mathbb{E}\left[\abs{\tilde{Y}_{s+K\Delta/\epsilon}^\epsilon-y_s^{K\Delta,X_{K \Delta}^\epsilon,Y_{K \Delta}^\epsilon}}^2\right] \\
& = 2 \mathbb{E}\left[ <B(K\Delta,X_{K\Delta}^\epsilon,\tilde{Y}^\epsilon_{s+K\Delta / \epsilon} )-B(K\Delta,X_{K \Delta}^\epsilon,y_s^{K\Delta,X_{K \Delta}^\epsilon,Y_{K \Delta}^\epsilon}),\tilde{Y}_{s+K\Delta /\epsilon}^\epsilon-y_s^{K\Delta,X_{K \Delta}^\epsilon,Y_{K \Delta}^\epsilon} >\right] \\
&+ 2 \epsilon^{1-\eta} \mathbb{E}\left[ <D(K\Delta,X_{K\Delta}^\epsilon,\tilde{Y}^\epsilon_{s+K\Delta / \epsilon} ),\tilde{Y}_{s+K\Delta /\epsilon}^\epsilon-y_s^{K\Delta,X_{K \Delta}^\epsilon,Y_{K \Delta}^\epsilon} >\right]  \\
& + \mathbb{E}\left[ \abs{C(K\Delta,X_{K\Delta}^\epsilon,\tilde{Y}^\epsilon_{s+K\Delta / \epsilon} )-C(K\Delta,X_{K \Delta}^\epsilon,y_s^{K\Delta,X_{K \Delta}^\epsilon,Y_{K \Delta}^\epsilon})}^2\right]  \\
& \leq -2 \beta \mathbb{E}\left[\abs{\tilde{Y}_{s+K\Delta/\epsilon}^\epsilon-y_s^{K\Delta,X_{K \Delta}^\epsilon,Y_{K \Delta}^\epsilon}}^2\right] 
+2 \epsilon^{1-\eta}\mathbb{E}\left[\abs{D(K\Delta,X_{K\Delta}^\epsilon,\tilde{Y}^\epsilon_{s+K\Delta / \epsilon} )}\abs{\tilde{Y}_{s+K\Delta /\epsilon}^\epsilon-y_s^{K\Delta,X_{K \Delta}^\epsilon,Y_{K \Delta}^\epsilon} }\right]  \\
& \leq - \beta \mathbb{E}\left[\abs{\tilde{Y}_{s+K\Delta/\epsilon}^\epsilon-y_s^{K\Delta,X_{K \Delta}^\epsilon,Y_{K \Delta}^\epsilon}}^2\right] + \frac{\epsilon^{1-\eta}}{\beta} \mathbb{E}\left[\abs{D(K\Delta,X_{K\Delta}^\epsilon,\tilde{Y}^\epsilon_{s+K\Delta / \epsilon} )}^2 \right] \\
& \leq - \beta \mathbb{E}\left[\abs{\tilde{Y}_{s+K\Delta/\epsilon}^\epsilon-y_s^{K\Delta,X_{K \Delta}^\epsilon,Y_{K \Delta}^\epsilon}}^2\right] + \epsilon^{1-\eta} C_1
\end{align*}
\end{small}
where $C_1$ is independent of $\epsilon$ and $K$.\\
Here we have used the fact that $\mathbb{E} [|\tilde{Y}^\epsilon_{s+K\Delta / \epsilon}|^2]\leq \overline{C}$ for $\overline{C}$ independent of $\epsilon$ and $K$.\\
Then by the comparison theorem we obtain:
\begin{equation*}
    \mathbb{E}\left[\abs{\tilde{Y}_{s+K\Delta/\epsilon}^\epsilon-y_s^{K\Delta,X_{K \Delta}^\epsilon,Y_{K \Delta}^\epsilon}}^2\right] \leq C_2 \epsilon^{1-\eta}
\end{equation*}
for every $s\in [0,\Delta /\epsilon)$ and $C_2$ independent of $\epsilon$ and $K$.\\
This implies \eqref{eq:assumption_tilde_frozen}.

\textit{Step 3}: we prove that the conclusion of theorem \ref{th:averaging} hold.
The proof is the same of the one for theorem \ref{th:averaging}, what changes is the following:\\
limits \eqref{eq1_th} and \eqref{eq2_th} follow by a similar calculation as before.\\
Then \eqref{eq3_proof} follows by the splitting
\begin{align*}
   & \mathbb{E} \left [ \abs{\mathbb{E} \left [ \sum_{K=\lceil t_1 / \Delta \rceil }^{\lceil t_2/ \Delta \rceil -1} \mathbb{E} \left [ \epsilon \int_{K \Delta / \epsilon}^{(K+1) \Delta / \epsilon  } b_i(K\Delta,x_{K \Delta}^{\epsilon},\tilde{Y}_s^{\epsilon})-\overline{b}_i(K\Delta,x_{K \Delta}^{\epsilon})ds \bigg| \mathfrak{F}'_{K \Delta}  \right ] \bigg| \mathfrak{F}'_{t_1}  \right ]} \right ] \\
   &\leq \mathbb{E} \left [ \abs{\mathbb{E} \left [ \sum_{K=\lceil t_1 / \Delta \rceil }^{\lceil t_2/ \Delta \rceil -1} \mathbb{E} \left [ \epsilon \int_{0}^{ \Delta / \epsilon  } b_i(K\Delta,x_{K \Delta}^{\epsilon},\tilde{Y}_{s+K \Delta/\epsilon}^{\epsilon})-b_i(K\Delta,x_{K \Delta}^{\epsilon},y_s^{(K\Delta,x_{K \Delta}^{\epsilon},Y_{K \Delta}^{\epsilon})})ds \bigg| \mathfrak{F}'_{K \Delta}  \right ] \bigg| \mathfrak{F}'_{t_1}  \right ]} \right ]\\
   &+\mathbb{E} \left [ \abs{\mathbb{E} \left [ \sum_{K=\lceil t_1 / \Delta \rceil }^{\lceil t_2/ \Delta \rceil -1} \mathbb{E} \left [ \epsilon \int_{0}^{\Delta / \epsilon  } b_i(K\Delta,x_{K \Delta}^{\epsilon},y_s^{(K\Delta,x_{K \Delta}^{\epsilon},Y_{K \Delta}^{\epsilon})})-\overline{b}_i(K\Delta,x_{K \Delta}^{\epsilon})ds \bigg| \mathfrak{F}'_{K \Delta}  \right ] \bigg| \mathfrak{F}'_{t_1}  \right ]} \right ]
\end{align*}
The second term on the right hand side tends to zero as before.\\
Now also the first term on the right hand side tends to zero:
\begin{align*}
    &\mathbb{E} \left [ \abs{\mathbb{E} \left [ \sum_{K=\lceil t_1 / \Delta \rceil }^{\lceil t_2/ \Delta \rceil -1} \mathbb{E} \left [ \epsilon \int_{0}^{ \Delta / \epsilon  } b_i(K\Delta,x_{K \Delta}^{\epsilon},\tilde{Y}_{s + K \Delta/\epsilon}^{\epsilon})-b_i(K\Delta,x_{K \Delta}^{\epsilon},y_s^{(K\Delta,x_{K \Delta}^{\epsilon},Y_{K \Delta}^{\epsilon})})ds \bigg| \mathfrak{F}'_{K \Delta}  \right ] \bigg| \mathfrak{F}'_{t_1}  \right ]} \right ]\\
    & \leq  \sum_{K=\lceil t_1 / \Delta \rceil }^{\lceil t_2/ \Delta \rceil -1}  \epsilon \int_{0}^{\Delta / \epsilon  } L \mathbb{E} \left [  \abs{\tilde{Y}_{s + K \Delta/\epsilon}^{\epsilon}-y_s^{(K\Delta,x_{K \Delta}^{\epsilon},Y_{K \Delta}^{\epsilon})}} \right ] ds    
\end{align*}
which tends to zero as $\epsilon \to 0$ by \eqref{eq:assumption_tilde_frozen}.\\
Finally \eqref{eq_aux_3} and \eqref{eq_aux} follow by a similar argument. This concludes the proof.
\end{proof}
\begin{remark}
With similar ideas we can perturb also the diffusion term of the fast equation with a slower term of order $\epsilon^{-\theta}$, $0 \leq \theta < 1/2$.
\end{remark}

\section{Financial Application: Local Stochastic Volatility Model}
\label{sec:Financial Application}
\subsection{Objective Probability}
Under the objective probability measure $P$ we consider a local stochastic volatility model defined by the following system of stochastic differential equations dependent on a small parameter $\epsilon \in (0,1]$
\begin{equation} \label{eq:system_loc_stoch_vol}
\begin{cases}
dS_t^{\epsilon}=\mathcal{H}(t,S_t^\epsilon,Y_t^\epsilon) S_t^\epsilon dt+ \mathcal{F}(t,S_t^{\epsilon},Y_t^\epsilon)  S_t^\epsilon dW_t \\
S_0^{\epsilon}=s_0 >0\\
dY_t^{\epsilon}=\epsilon^{-1} \mathcal{B}(t,S_t^{\epsilon},Y_t^\epsilon) dt+\epsilon^{-1/2} \mathcal{C}(t,S_t^{\epsilon},Y_t^\epsilon)  d\tilde{W}_t \\
Y_0^{\epsilon}=y_0 \in \mathbb{R}
\end{cases}
\end{equation}
where $0 \leq t \leq T$, $\{S_t^\epsilon\}_{t \geq 0}$ is the stock price, the real function $\mathcal{H}(t,s,y)$ is the local stochastic drift, $\mathcal{F}(t,s,y)$ is the local stochastic volatility, $\mathcal{B}(t,s,y)$ and $\mathcal{C}(t,s,y)$ are the coefficients of the process $\{Y_t^\epsilon\}_{t \geq 0}$ and $\{W_t\}_{t \geq 0}$ and $\{\tilde{W}_t\}_{t \geq 0}$ are Brownian motions with instantaneous correlation $\rho \in (-1,1)$. The parameter $\epsilon$ represents the ratio of time scales between the stock price and the stochastic volatility.\\
We assume that $$\tilde{W}_t= \rho W_t + \sqrt{1-\rho^2}Z_t$$ where $\{Z_t\}_{t \geq 0}$ is a Brownian motion independent of $\{W_t\}_{t \geq 0}$ so that $corr(W_t,\tilde{W}_t)=\rho$.

To apply the theory of the previous section it is convenient to consider the log-price $X_t^\epsilon=\log(S_t^\epsilon)$. By Ito's formula we have:
\begin{equation} \label{system_example}
\begin{cases}
dX_t^{\epsilon}=(H(t,X_t^{\epsilon},Y_t^\epsilon)- F(t,X_t^{\epsilon},Y_t^\epsilon)^2/2) dt+ F(t,X_t^{\epsilon},Y_t^\epsilon) dW_t \\
X_0^{\epsilon}=x_0 \\
dY_t^{\epsilon}=\epsilon^{-1} B(t,X_t^{\epsilon},Y_t^\epsilon)dt+\epsilon^{-1/2}  C(t,X_t^{\epsilon},Y_t^\epsilon)  d\tilde{W}_t \\
Y_0^{\epsilon}=y_0 
\end{cases}
\end{equation}
where $x_0 = \log(s_0)$,
$H(t,x,y)=\mathcal{H}(t,e^x,y)$, $F(t,x,y)=\mathcal{F}(t,e^x,y)$,  $B(t,x,y)=\mathcal{B}(t,e^x,y)$ and $C(t,x,y)=\mathcal{C}(t,e^x,y)$.\\
We assume that 
\begin{align*}
& |H(t,x,y)|\leq M \\
& 0<M_1\leq F(t,x,y) \leq M_2 \\
& |B(t,x,y)| \leq M (1+|x|+|y|) \\
& |C(t,x,y)| \leq M
\end{align*}
for $M>0$, $M_1>0$, $M_2>0$,  $0 \leq t \leq T$, $x,y \in \mathbb{R}$,
\begin{align*} 
&\abs{H(t,x_2,y_2)-H(s,x_1,y_1)}  \leq L (|t-s|^{\gamma_0}+|x_2-x_1|+|y_2-y_1|) \\
&\abs{F(t,x_2,y_2)-F(s,x_1,y_1)}  \leq L (|t-s|^{\gamma_0}+|x_2-x_1|+|y_2-y_1|) \\
& \abs{B(t,x_2,y_2)-B(s,x_1,y_1)}  \leq L (|t-s|^{\gamma_0}+|x_2-x_1|+|y_2-y_1|) \\
& \abs{C(t,x_2,y_2)-C(s,x_1,y_1)}  \leq L (|t-s|^{\gamma_0}+|x_2-x_1|+|y_2-y_1|)
\end{align*}
for $L>0$, $\gamma_0>0$, $0 \leq t,s \geq T$, $x_1,x_2,y_1,y_2 \in \mathbb{R}$ and 
\begin{equation}\label{eq:dissipative_finance}
(B(t,x,y_2)-B(t,x,y_1))(y_2-y_1) + (C(t,x,y_2)-C(t,x,y_1))^2 \leq -\beta |y_2-y_1|^2
\end{equation}
for $\beta>0$,  $0 \leq t \leq T$, $x,y_1,y_2 \in \mathbb{R}$.
\begin{remark}
As we have seen in section \ref{sec:Explicit conditions} the dissipativity condition \eqref{eq:dissipative_finance} gives all the ergodic properties. It is satisfied for example by the coefficients of the Ornstein Uhlenbeck process in \cite{papanicolaou2}.
\end{remark}

Since system \eqref{system_example} satisfies hypotheses \ref{hp:lip}, \ref{hp:sub}, \ref{hp:pos}, \ref{hp:dis} then by theorem \ref{th:averaging} we have that $ X^{\epsilon} \to \overline{X}$ weakly in  $\mathcal{C}[0,T]$ as $\epsilon \to 0$ where the process $\overline{X}$ is the unique strong solution of the averaged equation
\begin{align*}
d\overline{X} _t = \left (\overline{H}(t,\overline{X}_t)- \overline{F}(t,\overline{X}_t)^2/2 \right)dt+\overline{F}(t,\overline{X}_t)dW_t && \overline{X}_0=x_0, \forall t \in[0,T]
\end{align*}
Here we have
$$\overline{H}(t,x)=\int_{\mathbb{R}} H(t,x,z)\mu^{t,x}(dz)$$
$$\overline{F}(t,x)=\sqrt{\int_{\mathbb{R}} F(t,x,z)^2\mu^{t,x}(dz)}$$
where $\mu^{t,x}$ is the distribution of the unique invariant measure for the semigroup of the frozen equation $\{ \mathcal{P}_s^{t,x} \}_{s \geq 0}$ defined by $\mathcal{P}_s^{t,x} \phi(y)=\mathbb{E}[\phi(y_s^{t,x,y})]$, $y \in \mathbb{R}$, $\phi \in B_B(\mathbb{R})$
\begin{align*}
 d y_{s}^{t,x,y} =B\left(t,x,y_{s}^{t,x,y}\right)  d s+C\left(t,x,y_{s}^{t,x,y}\right )d \tilde{W}_{s} && y_{0}^{t,x,y}=y \in \mathbb{R}
\end{align*}
Then we define $\overline{S}_t=\exp(\overline{X}_t)$ and by Ito's formula we have:
\begin{align*}
d\overline{S} _t =  \overline{\mathcal{H}}(t,\overline{S} _t) \overline{S}_t  dt+\overline{\mathcal{F}}(t,\overline{S} _t) \overline{S}_t dW_t && \overline{S}_0=s_0, \forall t \in[0,T]
\end{align*}
where 
\begin{equation}\label{eq:averaged_drift}
\overline{\mathcal{H}}(t,s)=\overline{H}(t,\log(s))=\int_{\mathbb{R}} \mathcal{H}(t,s,z)\mu^{t,x}(dz)
\end{equation}
\begin{equation}\label{eq:averaged_volatility}
\overline{\mathcal{F}}(t,s)=\overline{F}(t,\log(s))=\sqrt{\int_{\mathbb{R}} \mathcal{F}(t,s,z)^2\mu^{t,x}(dz)}
\end{equation}
Since the family of processes $\{X^\epsilon\}_\epsilon$ converges weakly in $\mathcal{C}[0,T]$ to $\overline{X}$ then the family of processes $\{S^\epsilon=\exp(X^\epsilon)\}_\epsilon$ converges weakly in $\mathcal{C}[0,T]$ to the process $\overline{S}=\exp(\overline{X})$.\\
We have then the following theorem:
\begin{theorem}
The family $\{S^\epsilon\}_{\epsilon>0}$ converges weakly in $\mathcal{C}[0,T]$ as $\epsilon \to 0$ to the unique strong solution $\overline{S}$ of the averaged equation 
\begin{align*}
d\overline{S} _t =  \overline{\mathcal{H}}(t,\overline{S}_t) \overline{S}_t  dt+\overline{\mathcal{F}}(t,\overline{S} _t) \overline{S}_t dW_t && \overline{S}_0=s_0, \forall t \in[0,T]
\end{align*}
\end{theorem}
Let's notice that the limit model for the stock price $\{ \overline{S}_t \}_{t \geq 0}$ is a local volatility model with local drift $\overline{\mathcal{H}}(t,s)$ and local volatility $\overline{\mathcal{F}}(t,s)$.

\subsection{Risk Neutral Pricing}
Now as in \cite{papanicolaou2} we construct a family of risk neutral measures by means of Girsanov theorem:
indeed let the short rate $\mathcal{R}(t,S_t^\epsilon,Y_t^\epsilon)$ where $\mathcal{R}(t,s,y)$ is a deterministic function. Then set $r(t,x,y)=\mathcal{R}(t,\log(s),y)$. Assume that $r(t,x,y)$ is a bounded and
\begin{align*}
\abs{r(t,x_2,y_2)-r(s,x_1,y_1)}  \leq L (|t-s|^{\gamma_0}+|x_2-x_1|+|y_2-y_1|)
\end{align*}
for every $0 \leq t,s \leq T$, $x_1,x_2,y_1,y_2 \in \mathbb{R}$.\\
Let $\theta_t=(H(t,X_t^{\epsilon},Y_t^{\epsilon})-r(t,X_t^{\epsilon},Y_t^{\epsilon}))/F(t,X_t^{\epsilon},Y_t^{\epsilon})$ and define:
\begin{equation*}
    W_t^*=W_t+\int_0^t \theta_s ds
\end{equation*}

\begin{equation*}
    Z_t^*=Z_t+\int_0^t \gamma_s ds
\end{equation*}
for an arbitrary adapted bounded $\gamma_t$.\\
Now by Girsanov theorem $W^*$ and $Z^*$ are independent Brownian motions under a measure $Q^\gamma$ defined by:

\begin{equation*}
    dQ^\gamma=\exp\left( \int_0^T \theta_s dW_s + \int_0^T \gamma_s dZ_s -1/2 \int_0^T \theta_s^2 + \gamma_s^2 ds\right) dP
\end{equation*}
Then $\{Q^\gamma\}_{\gamma}$ is a family of equivalent risk-neutral martingale measures parametrized by $\gamma$ which is called market price of volatility risk. We assume that $\gamma_t=\gamma(t,X_t^{\epsilon},Y_t^{\epsilon})$ and
\begin{align*}
\abs{\gamma(t,x_2,y_2)-\gamma(s,x_1,y_1)}  \leq L (|t-s|^{\gamma_0}+|x_2-x_1|+|y_2-y_1|)
\end{align*}
for every $0 \leq t,s \leq T$, $x_1,x_2,y_1,y_2 \in \mathbb{R}$.\\
Then under $Q^\gamma$ \eqref{system_example} becomes:
\begin{equation} \label{system_example2}
\begin{cases}
dX_t^{\epsilon}=(r(t,X_t^{\epsilon},Y_t^\epsilon)- F(t,X_t^{\epsilon},Y_t^\epsilon)^2/2) dt+ F(t,X_t^{\epsilon},Y_t^\epsilon) dW^*_t \\
X_0^{\epsilon}=x_0 \\
dY_t^{\epsilon}= \left[\epsilon^{-1} \left(B(t,X_t^{\epsilon},Y_t^\epsilon)\right) -  \epsilon^{-1/2}  C(t,X_t^{\epsilon},Y_t^\epsilon) \Lambda(t,X_t^{\epsilon},Y_t^{\epsilon}) \right]dt+\epsilon^{-1/2}  C(t,X_t^{\epsilon},Y_t^\epsilon)  d\tilde{W}^*_t \\
Y_0^{\epsilon}=y_0 
\end{cases}
\end{equation}
where $$\tilde{W}_t^*= \rho W_t^* + \sqrt{1-\rho^2}Z_t^*$$ and $$\Lambda(t,X_t^{\epsilon},Y_t^{\epsilon}) = \rho (H(t,X_t^{\epsilon},Y_t^{\epsilon})-r(t,X_t^{\epsilon},Y_t^{\epsilon}))/F(t,X_t^{\epsilon},Y_t^{\epsilon})+ \sqrt{1-\rho^2}\gamma(t,X_t^{\epsilon},Y_t^{\epsilon})$$
or equivalently:
\begin{equation} \label{system_example3}
\begin{cases}
dS_t^{\epsilon}=\mathcal{R}(t,S_t^{\epsilon},Y_t^\epsilon) S_t^\epsilon dt+ \mathcal{F}(t,S_t^{\epsilon},Y_t^\epsilon)  S_t^\epsilon dW^*_t \\
S_0^{\epsilon}=s_0 >0\\
dY_t^{\epsilon}= \left[\epsilon^{-1} \mathcal{B}(t,S_t^{\epsilon},Y_t^\epsilon) -  \epsilon^{-1/2}  \mathcal{C}(t,S_t^{\epsilon},Y_t^\epsilon) \lambda(t,S_t^{\epsilon},Y_t^{\epsilon}) \right]dt+\epsilon^{-1/2}  \mathcal{C}(t,S_t^{\epsilon},Y_t^\epsilon)  d\tilde{W}^*_t \\
Y_0^{\epsilon}=y_0 
\end{cases}
\end{equation}
where $\lambda(t,s,y)=\Lambda(t,\log(s),y)$.

Now let $0 \leq \tau \leq T$ and consider a financial European option with maturity $T$. A derivative of this kind is defined by its non negative payoff $\Psi=\Psi(s)$
which prescribes the value of the contract at its maturity.\\
Under the risk neutral measure $Q^\gamma$ we consider the fair price $P^\epsilon_\tau$ of the contract at time $\tau$. This is calculated as:
\begin{align*}
P^\epsilon_\tau= \mathbb{E}^*\left [ e^{-\int_{\tau}^T \mathcal{R}(s,S_s^{\epsilon},Y_s^\epsilon) ds} \Psi\left(S^\epsilon_T \right)  \big| \mathcal{F}_\tau \right ] 
\end{align*}
where $\mathcal{F}_\tau= \sigma(W_t^*,Z_t^*, 0 \leq t \leq \tau )$.\\
Applying the theory of section \ref{sec:Explicit conditions} we can show the following theorem:
\begin{theorem}\label{th_example}
Under the risk neutral measure $Q^\gamma$ the family $\{S^\epsilon\}_{\epsilon>0}$ converges weakly in $\mathcal{C}[0,T]$ as $\epsilon \to 0$ to the solution $\overline{S}$ of the averaged equation
\begin{align}\label{dynamic_limit}
d \overline{S}_t=\overline{\mathcal{R}}(t,\overline{S}_t) \overline{S}_t dt+ \overline{\mathcal{F}}(t,\overline{S}_t)\overline{S}_t dW_t^*  && \overline{S}_0=s_0, t \geq 0
\end{align}
where 
\begin{align*}
\overline{\mathcal{R}}(t,s)=\int_{\mathbb{R}} \mathcal{R}(t,s,z)\mu^{t,x}(dz)
&&
\overline{\mathcal{F}}(t,s)=\sqrt{\int_{\mathbb{R}} \mathcal{F}(t,s,z)^2\mu^{t,x}(dz)}
\end{align*}
Moreover let $\Psi \in \mathcal{C}_B(\mathbb{R})$ and $0 \leq \tau \leq T$ then   
\begin{align}\label{eq_thesis_ex}
\lim_{\epsilon \to 0} P^\epsilon_\tau=\overline{P}_\tau
\end{align}
where 
$$\overline{P}_\tau=\mathbb{E}^* \left [ e^{-\int_{\tau}^T \overline{\mathcal{R}}(s,\overline{S}_s) ds} \Psi(\overline{S}_T)  \big| \mathcal{F}_\tau  \right] $$ 
\end{theorem}
Let's notice that again the limit model for the stock price $\{ \overline{S}_t \}_{t \geq 0}$ is a local volatility model with $\overline{\mathcal{R}}(t,s)$, $\overline{\mathcal{F}}(t,s)$.
\begin{proof}
Set $D(t,x,y)=C(t,x,y)\Lambda(t,x,y)$ and notice that $F^2(t,x,y)$ and $D(t,x,y)$ are bounded and satisfy  
\begin{align*} 
&\abs{F^2(t,x_2,y_2)-F^2(s,x_1,y_1)}  \leq \overline{C} (|t-s|^{\gamma_0}+|x_2-x_1|+|y_2-y_1|) \\
&\abs{D(t,x_2,y_2)-D(s,x_1,y_1)}  \leq \overline{C} (|t-s|^{\gamma_0}+|x_2-x_1|+|y_2-y_1|) 
\end{align*}
for $\overline{C}>0$.\\
Therefore we can apply proposition \ref{prop:extension2} to system \eqref{system_example2} under the risk neutral measure $Q^\gamma$. Then the family of processes $\{X^\epsilon\}_\epsilon$ converges weakly in $\mathcal{C}[0,T]$ to the solution $\overline{X}$ of
\begin{align*}
d\overline{X} _t = \left (\overline{r}(t,\overline{X}_t)- \overline{F}(t,\overline{X}_t)^2/2 \right)dt+\overline{F}(t,\overline{X}_t))dW^*_t && \overline{X}_0=x_0, \forall t \in[0,T]
\end{align*}
This implies that the family of processes $\{S^\epsilon=\exp(X^\epsilon)\}_\epsilon$ converges weakly in $\mathcal{C}[0,T]$ to the process $\overline{S}=\exp(\overline{X})$ satisfying \eqref{dynamic_limit}.\\
When the short rate is deterministic we can proceed in this way: by the Markov property we have
\begin{align}\label{eq:markov_price}
P^\epsilon_\tau= \mathbb{E}^*\left [ e^{-\int_{\tau}^T \mathcal{R}(s) ds} \Psi\left(S^\epsilon_T \right)  \big| S_\tau^\epsilon,Y_\tau^\epsilon \right ] 
\end{align}
By the weak convergence of $S^\epsilon$ it follows:
\begin{align*}
\lim_{\epsilon \to 0} P^\epsilon_\tau=\mathbb{E}^* \left [ e^{-\int_{\tau}^T \mathcal{R}(s) ds} \Psi(\overline{S}_T)  \big| \overline{S}_\tau  \right]
\end{align*}
Then by the Markov property we have \eqref{eq_thesis_ex}.

If however the short rate is not deterministic we need to consider the integral process
\begin{equation}
    I_t^\epsilon=\int_0^t \mathcal{R}(s,S_s^{\epsilon},Y_s^\epsilon) ds
\end{equation}
and consider $(S^{\epsilon},I^\epsilon,Y^\epsilon)$ under the risk neutral measure. Then $(S^{\epsilon},I^\epsilon)$ converges weakly in $\mathcal{C}[0,T]^2$ to $(\overline{S},\overline{I})$ where the process $\overline{I}$ is defined by 
\begin{align*}
    d\overline{I}_t= \overline{\mathcal{R}}(t,\overline{S}_t) dt && \overline{I}_0=0
\end{align*}
Now we have 
\begin{align}
P^\epsilon_\tau= \mathbb{E}^* \left [ e^{-(I^\epsilon_T-I^\epsilon_\tau)} \Psi\left(S^\epsilon_T \right)  \big| \mathcal{F}_\tau \right ] 
\end{align}
As $(S^{\epsilon},I^\epsilon,Y^\epsilon)$ is Markovian:
\begin{align}
P^\epsilon_\tau= \mathbb{E}^* \left [ e^{-(I^\epsilon_T-I^\epsilon_\tau)} \Psi\left(S^\epsilon_T \right)  \big| S_\tau^\epsilon, I_\tau^\epsilon, Y_\tau^\epsilon \right ] 
\end{align}
Then by the weak convergence of $(S^{\epsilon},I^\epsilon)$ we have:
\begin{align*}
\lim_{\epsilon \to 0} P^\epsilon_\tau & =\mathbb{E}^* \left [ e^{-(\overline{I}_T-\overline{I}_\tau)} \Psi(\overline{S}_T)  \big| \overline{S}_\tau, \overline{I}_\tau  \right] \\
& =\mathbb{E}^* \left [ e^{-(\overline{I}_T-\overline{I}_\tau)} \Psi(\overline{S}_T)  \big| \mathcal{F}_\tau  \right]
\end{align*}
which is \eqref{eq_thesis_ex}.
\end{proof}
\subsection{Path-dependent Options}
Now we treat path-dependent options. We follow a similar approach to the one in \cite{fuhrman} by setting the problem in an infinite dimensional setting which permits us to take into account the whole trajectory of the stock price.\\
We assume for simplicity that the short rate is deterministic $\mathcal{R}(t)$ but the approach works also in the general case $\mathcal{R}(t,S_t^{\epsilon},Y_t^\epsilon)$ by considering the agumented system with integral process $I^\epsilon$ as we have just done.\\
For $t \in [0,T]$ under the risk-neutral measure $Q^\gamma$ consider the vector process
\begin{equation}
    \mathcal{S}_t^\epsilon = \left[S_t^\epsilon,S_t^\epsilon(\cdot) \right]^T
\end{equation}
where the notation $S^\epsilon_t(\cdot)$ denotes the trajectory $S^\epsilon_t(\theta)=S^\epsilon_{t+\theta}$, $\theta \in [-T,0]$.
Here we set $S^\epsilon_\theta=\nu_0(\theta)$ where $\nu_0 \in \mathcal{C}$[-T,0] is a given function.\\
Then by \cite{delay} the couple $\left( \mathcal{S}_t^\epsilon,Y_t^\epsilon \right)$ solves an infinite dimensional stochastic differential equation on the Hilbert space $H= \mathbb{R} \times \mathcal{L}^2[-T,0] \times \mathbb{R}$ and it is a Markov process thanks to the regularity of the coefficients of this equation. Moreover we set $H^1=\mathbb{R} \times \mathcal{L}^2[-T,0]$. 

Now consider a bounded path-dependent option with exercise date $T$ of the form 
\begin{align*}
 \Psi(S_T^\epsilon,S_T^\epsilon(\cdot)) && 
\end{align*}
where $\Psi \colon H^1 \rightarrow \mathbb{R}$ is bounded and continuous.\\
Under the risk neutral measure $Q^\gamma$ we consider the fair price $P^\epsilon_\tau$ of the contract at time $0\leq\tau  \leq T$. This is calculated as:
\begin{align*}
P^\epsilon_\tau= \mathbb{E}^* \left [ e^{-\int_{\tau}^T \mathcal{R}(s) ds} \Psi(S_T^\epsilon,S_T^\epsilon(\cdot))  \big| \mathcal{F}_\tau \right ] 
\end{align*}
where 
$$\mathcal{F}_\tau= \sigma(W_t^*,Z_t^*, 0 \leq t \leq \tau )$$
By the Markov property we have:
\begin{align*}
P^\epsilon_\tau= \mathbb{E}^* \left [ e^{-\int_{\tau}^T \mathcal{R}(s) ds} \Psi(S_T^\epsilon,S_T^\epsilon(\cdot))  \big| S_\tau^\epsilon,Y_\tau^\epsilon,S_\tau^\epsilon(\cdot) \right ] 
\end{align*}
Since $S_T^\epsilon(\cdot)$ is continuous, $S_T^\epsilon(0)=S_T^\epsilon$ and $\mathcal{C}[-T,0] \subset \mathcal{L}^2[-T,0]$ we can consider the restriction of the operator $\Psi$ to the set of continuous functions in the following way:
\begin{align*}
    & \Psi^R \colon \mathcal{C}[-T,0] \rightarrow \mathbb{R} \\
    & \Psi^R(\eta)=\Psi(\eta(0),\eta)
\end{align*}
Since $\Psi$ is bounded and continuous it is possible to show that $\Psi^R$ is also bounded and continuous.\\
Then we have:
\begin{align*}
P^\epsilon_\tau= \mathbb{E}^* \left [ e^{-\int_{\tau}^T \mathcal{R}(s) ds} \Psi^R(S_T^\epsilon(\cdot))  \big| S_\tau^\epsilon,Y_\tau^\epsilon,S_\tau^\epsilon(\cdot) \right ] 
\end{align*}
By the weak convergence of $S^\epsilon$ in the space of continuous functions  we have:
\begin{align*}
\lim_{\epsilon \to 0} P^\epsilon_\tau & =\mathbb{E}^* \left [ e^{-\int_{\tau}^T \mathcal{R}(s) ds}\Psi^R(\overline{S}_T(\cdot))   \big| \overline{S}_\tau,\overline{S}_\tau(\cdot) \right ]  \\
& =\mathbb{E}^* \left [ e^{-\int_{\tau}^T \mathcal{R}(s) ds}\Psi(\overline{S}_T,\overline{S}_T(\cdot))   \big| \overline{S}_\tau,\overline{S}_\tau(\cdot) \right ]
\end{align*}
Now for $t \in [0,T]$ let's consider the vector process
\begin{equation*}
    \overline{\mathcal{S}}_t = \left[\overline{S}_t,\overline{S}_t(\cdot) \right]^T
\end{equation*}
Here we set $\overline{S}_\theta=\nu_0(\theta)$.\\
As before by \cite{delay} $\overline{\mathcal{S}}_t$ solves an infinite dimensional stochastic differential equation on the Hilbert space $H^1= \mathbb{R} \times \mathcal{L}^2[-T,0]$ and it is a Markov process.
Therefore by the Markov property we have the following theorem:
\begin{theorem}
Let $\Psi \colon H^1 \rightarrow \mathbb{R}$ be bounded and continuous. Then
\begin{align*}
\lim_{\epsilon \to 0} P^\epsilon_\tau=\overline{P}_\tau
\end{align*}
where 
\begin{align*}\overline{P}_\tau =\mathbb{E}^* \left [ e^{-\int_{\tau}^T \mathcal{R}(s) ds}\Psi(\overline{S}_T,\overline{S}_T(\cdot))   \big| \mathcal{F}_\tau \right ]
\end{align*}
\end{theorem}

This result has some limitations: even if the functionals defined on $\mathbb{R}\times \mathcal{L}^2$ permit to treat Asian options, they exclude many path-dependent derivatives, for example look-back options which are functions of the supremum (or infimum) of the stock price with respect to time. Therefore we want to set a strategy to cover also these derivatives: we start by considering look-back options of the form
\begin{align*}
    \Psi(S_T^\epsilon(\cdot))=\Psi^0 \left (S_T^\epsilon,\sup_{\theta \in [-T,0]} a(\theta)S_{T+\theta}^\epsilon \right) 
\end{align*}
where $\Psi^0 \colon \mathbb{R}^2 \rightarrow \mathbb{R}$ is a continuous bounded function and $a \in \mathcal{C}[-T,0]$.\\
We consider as usual the fair price $P^\epsilon_\tau$ of the contract at time $0\leq\tau  \leq T$ by
\begin{align*}
P^\epsilon_\tau= \mathbb{E}^*\left [  e^{-\int_{\tau}^T \mathcal{R}(s) ds} \Psi(S_T^\epsilon(\cdot))  \big| \mathcal{F}_\tau \right ] 
\end{align*}
Now we set
\begin{align*}
    \Psi(\eta)=\Psi^0 \left (\eta(0),\sup_{\theta \in [-T,0]} a(\theta)\eta(\theta) \right) && \eta \in \mathcal{C}[-T,0]
\end{align*}
Then for $\delta >0$ we define $\Psi^\delta \colon H^1 \rightarrow \mathbb{R}$ in the following way:
\begin{align*}
    \Psi^\delta(\phi,\eta)=\Psi^0 \left (\phi,\sup_{\theta \in [-T,0]} \eta^\delta(\theta) \right) && \phi \in \mathbb{R},\eta \in \mathcal{L}^2[-T,0]
\end{align*}
where 
\begin{align*}
    \eta^\delta(\theta)=\frac{1}{\delta} \int_{(\theta-\delta)^+}^\theta a(\xi)\eta(\xi)d\xi
\end{align*}
and we notice that
\begin{align*}
    \lim_{\delta \to 0} \Psi^\delta(\eta(0),\eta)=\Psi(\eta) 
\end{align*}
for every $\eta \in \mathcal{C}[-T,0]$.\\
This implies:
\begin{align*}
    \lim_{\delta \to 0} \Psi^\delta(S_T^\epsilon,S_T^\epsilon(\cdot))=\Psi(S_T^\epsilon(\cdot)) && a.s.
\end{align*}
Since $\Psi^\delta$ is bounded we have the Markov property:
\begin{align*}
\mathbb{E}^* \left [ \Psi^\delta(S_T^\epsilon,S_T^\epsilon(\cdot))  \big| \mathcal{F}_\tau \right ] =  \mathbb{E}^* \left [ \Psi^\delta(S_T^\epsilon,S_T^\epsilon(\cdot))  \big| S_\tau^\epsilon,Y_\tau^\epsilon,S_\tau^\epsilon(\cdot) \right ] 
\end{align*}
Then by dominated convergence we have:
\begin{align*}
\mathbb{E}^* \left [ \Psi(S_T^\epsilon(\cdot))  \big| \mathcal{F}_\tau \right ] =  \mathbb{E}^* \left [ \Psi(S_T^\epsilon(\cdot))  \big| S_\tau^\epsilon,Y_\tau^\epsilon,S_\tau^\epsilon(\cdot) \right ] 
\end{align*}
This implies:
\begin{align*}
P^\epsilon_\tau= \mathbb{E}^* \left [ e^{-\int_{\tau}^T \mathcal{R}(s) ds} \Psi(S_T^\epsilon(\cdot))  \big| S_\tau^\epsilon,Y_\tau^\epsilon,S_\tau^\epsilon(\cdot) \right ] 
\end{align*}
As before by the weak convergence of $S^\epsilon$ in the space of continuous functions we have:
\begin{align*}
\lim_{\epsilon \to 0} P^\epsilon_\tau=\overline{P}_\tau
\end{align*}
where
\begin{align*}\overline{P}_\tau =\mathbb{E}^* \left [ e^{-\int_{\tau}^T \mathcal{R}(s) ds}\Psi(\overline{S}_T(\cdot))   \big| \mathcal{F}_\tau \right ]
\end{align*}

This approach works in general for path-dependent derivatives defined by functionals on the space $\mathcal{C}[-T,0]$ which have a regular approximation through functionals on $\mathbb{R} \times \mathcal{L}^2[-T,0]$, indeed we have:
\begin{theorem}
Let $\Psi \colon \mathcal{C}[-T,0] \rightarrow \mathbb{R}$ be bounded and continuous. Moreover assume there exists a family of uniformly bounded functionals  $\{\Psi^\delta\}_{\delta>0}$, $\Psi^\delta \colon H^1 \rightarrow \mathbb{R}$ such that 
\begin{align*}
    \lim_{\delta \to 0} \Psi^\delta(\eta(0),\eta)=\Psi(\eta) && \forall \eta \in \mathcal{C}[-T,0]
\end{align*}
Then with the usual notations we have:
\begin{align*}
\lim_{\epsilon \to 0} P^\epsilon_\tau=\overline{P}_\tau
\end{align*}
\end{theorem}

\section*{Acknowledgements}
The author would like to thank his PhD's supervisors Giuseppina Guatteri and Gianmario Tessitore for their constant support. He also wants to thank Sandra Cerrai and Carlo Sgarra for some useful conversations.

\end{document}